\documentclass[10pt]{article}
\usepackage{algorithmic}
\usepackage{algorithm}
\usepackage{multirow}
\usepackage{xcolor}
\usepackage{enumitem}
\usepackage{amsfonts}
\usepackage{amsmath}
\usepackage{graphicx}
\usepackage{subfigure}
\usepackage{fancyhdr}
\usepackage{amsthm}
\usepackage{amssymb}
\usepackage{amscd}
\usepackage{amsmath}
\usepackage{float}
\newtheorem{problem}{Problem}[section]

\newtheorem{theorem}{Theorem}[section]

\newtheorem{lemma}{Lemma}[section]

\newtheorem{remark}{Remark}[section]

\begin{document}
\title{{\LARGE\bf Variational and numerical analysis of a dynamic viscoelastic contact problem with friction and wear}\thanks{This work was supported by  the National Natural Science Foundation of China (11471230, 11671282, 11771067) and the Applied Basic Project of Sichuan Province (2019YJ0204).}}
\author{{Tao Chen$^a$ Nan-jing Huang$^a$\thanks{Corresponding author,  E-mail: nanjinghuang@hotmail.com; njhuang@scu.edu.cn} \, and Yi-bin Xiao$^b$}\\
{\small\it a. Department of Mathematics, Sichuan University, Chengdu, Sichuan 610064, P.R. China}\\
{\small\it b. School of Mathematical Sciences, University of Electronic Science and Technology of China,}\\
{\small\it Chengdu, Sichuan, 611731, P. R. China}}
\date{}
\maketitle
\vspace*{-9mm}
\begin{center}
\begin{minipage}{5.8in}
{\bf Abstract.}
In this paper, we consider a dynamic viscoelastic contact problem with friction and wear, and describe it as a system of nonlinear partial differential equations. We formulate the previous problem as a hyperbolic quasi-variational inequality by employing the variational method. We adopt the Rothe method to show the existence and uniqueness of weak solution for the hyperbolic quasi-variational inequality under mild conditions. We also give a fully discrete scheme for solving the hyperbolic quasi-variational inequality and obtain error estimates for the fully discrete scheme.
\\ \ \\
\end{minipage}
\end{center}

\section{Introduction}
Viscoelastic contact is a well-discussed physical phenomenon which describes the deformation process of a viscoelastic body when it contacts with a rigid foundation. Various theoretical results and numerical algorithms with applications have been studied extensively for quasistatic viscoelastic contact problems in the literature (see, for example, \cite{Han 2001, Han book 2002,SX1,SXC2}).

In order to describe the process of deformation of a viscoelastic body with wear when it contacts with a rigid body foundation, several quasistatic viscoelastic frictional contact problems with wear were introduced and studied under different conditions; for instance, we refer the reader to \cite{chen 2001,Rochdi 1998,Shillor 2000} and the references therein. It is worth mentioning that Chen et al. \cite{chen 2001} were the first to derive error estimates of fully discrete schemes for solving quasistatic viscoelastic frictional contact problems with wear. Recently, Gasi\'{n}ski et al. \cite{Gasinski 2016} proposed a mathematical model to describe quasistatic frictional contact with wear between a thermoviscoelastic body and a moving foundation. Very recently, Jureczka and Ochal \cite{Jureczka 2018} obtained the numerical analysis and simulations for the quasistatic elastic frictional contact problem with wear.

It is well known that Duvaut and Lions \cite{Duvaut 1976} were the first to study quasi-static frictional contact viscoelastic problems within the framework of variational inequalities. From then on, various variational inequalities, hemivariational inequalities and other related problems have been derived from different physical phenomena in contact mechanics and abundant research results have been obtained for their studies (\cite{HanSonfonea 2000,Kinderlehrer 1980,Huang 2014,Sofonea 2009,SMX,SX3,SXC1,Stromberg 2008, XS3}). Recently, Mig\'{o}rski and Zeng \cite{zengshengda 2017} studied a class of hyperbolic variational inequalities and applied their results to study the existence of weak solutions for the dynamic frictional contact problem without wear.

As a generalization of the contact problem considered in \cite{chau 2004},  Chau et al. \cite{chau 2003} introduced and studied a dynamic frictionless contact problem and gave a fully discrete scheme for solving such problem. Bartosz \cite{Bartosz 2006} considered a dynamical viscoelastic contact problem to modify the model treated by Ciulcu et al. \cite{Ciulcu 2002}.  Especially, Bartosz \cite{Bartosz 2006} established the existence of weak solutions of the dynamical viscoelastic contact problem with wear by using the surjectivity result for a class of pseudomonotone operators in the framework of hemivariational inequalities.  Recently, Cocou \cite{cocou 2015} extended the static contact problem considered by Rabier et al. \cite{Rabier 2000} to a dynamic viscoelastic contact problem with friction and obtained an existence and uniqueness of the weak solution for such problem. However, to our best knowledge,  there is no study for the dynamic viscoelastic contact problem with friction and wear in the existing literature.  The motivation of this paper is to make a new attempt in this direction.

In this paper, we consider a mathematical model to describe a dynamic viscoelastic contact problem with friction and wear, in which the material behavior is followed by the Kelvin-Voigt viscoelastic constitutive law and the frictional contact is modelled with a wear governed by a simplified version of Archard's law \cite{Stromberg 2008, Stromberg 1996} for the velocity field associated to a version of Coulomb's law of dry friction. Compared with the model of Bartosz et al. \cite{Bartosz 2006},  the effect of friction on the contact boundary has been considered in our model in order to describe the frictional contact phenomenon. Moreover,  the method used in this paper is quite different from the one employed in \cite{Bartosz 2006}.  In fact, we establish the existence of weak solution by employing the Rothe method combining with the Banach fixed point theorem while Bartosz et al. \cite{Bartosz 2006} adopted the surjectivity result for multi-valued pseudomonotone operators.

The rest of this paper is organized as follows. Section 2 presents some necessary preliminaries  and the weak formulation of the dynamic viscoelastic contact problem with friction and wear. Inspired by Mig\'{o}rski and Zeng \cite{zengshengda 2017}, we prove the existence and uniqueness of the solution for the hyperbolic quasi-variational inequality under mild conditions by applying the Rothe method \cite{Kacur 1986} in Section 3.  We obtain the existence and uniqueness of weak solution for the dynamic viscoelastic contact problem with friction and wear in Section 4. Finally, we present the fully discrete scheme for solving the hyperbolic quasi-variational inequality and derive the error estimates for the fully discrete scheme in Section 5.

\section{Preliminaries}
\setcounter{equation}{0}

In this section, we first recall some notations that will be used later.  The main materials can be found in the book \cite{Roubicek 2013}.

Let $\mathbb{S}^d$ denote the space of second order symmetric tensors on $\mathbb{R}^d$ with $d=2,3$. Let $":"$ and $"\cdot"$ represent the inner product on $\mathbb{S}^d$ and $\mathbb{R}^d$, respectively, and $"|\cdot|"$ denotes the Euclidean norm on $\mathbb{S}^d$ and $\mathbb{R}^d$. We adopt the Einstein summation convention and two indicators separated by commas to indicate the partial derivative of the function corresponding to the first indicator respect to the second indicator. For example, $u_{i,j}\equiv \frac{\partial u_i}{\partial x_j}$ and $u_{ij}v_j\equiv \sum_{j=1}^du_{ij}\cdot v_j$. In what follows, we use the following notations:
\begin{equation*}\label{Hdefine}
H=L^2(\Omega,\mathbb{R}^d)=\{u=(u_i)| u_i \in L^2(\Omega)\}, \quad  Q=L^2(\Omega,\mathbb{S}^d)=\{\sigma = (\sigma_{ij})|  \sigma_{ij}=\sigma_{ji}\in L^2(\Omega)) \},
\end{equation*}
where $u$ is the displacement field in $\mathbb{R}^d$. For a smooth displacement field $u$, we write $\varepsilon $ and $Div$ for the meaning of deformation and divergence operators of $u$, respectively. Here, $\varepsilon $ and $Div$ are defined by
$$\varepsilon(u)=(\varepsilon_{ij}(u)),\quad \varepsilon_{ij}(u)=\frac{1}{2}(u_{i,j}+u_{j,i}),\quad Div\, \sigma =(\sigma_{ij,j}).$$
It is easy to see that $H$ and $Q$ are real Hilbert spaces endowed with inner products as follows:
$$(u,v)_H=\int_{\Omega}u\cdot vdx=\int_{\Omega}u_iv_idx,\quad (\sigma,\tau)_Q=\int_{\Omega}\sigma:\tau dx=\int_{\Omega}\sigma_{ij}\tau_{ij}dx.$$
Assuming $u$ has a partial derivative in a certain sense such as distribution sense, we can define a Hilbert space $ H_1=\{u\in H;\varepsilon(u)\in Q\}$ endowed with the inner product as follows:
$$(u,v)_{H_1}=(u,v)+(\varepsilon(u),\varepsilon(v))_Q.$$
To simplify notation,  let $\|\cdot \|_H$, $\|\cdot\|_Q$ and $\|\cdot\|_{H_1}$ denote the norms on spaces $H$, $Q$ and $H_1$, respectively.

We recall some spaces $W^{k,p}([0,T];X)$, $H^k([0,T];X)$ and $C([0,T]; X)$ for a Banach space $X$ equipped with the norm $\|\cdot\|_X$ for $1<p<\infty$ and $1\leq k$. Let $W^{k,p}([0,T];X)$ denote the space of all functions from $[0,T]$ to $X$ with the norm
\begin{equation*}
\|f\|_{W^{k,p}([0,T];X)}=\left\{
\begin{array}{ll}
\left(\int_0^T\sum_{1\leq l\leq k}\|\partial^l_tf \|_X^pdt\right)^{1/p}, & {1 \leq p<\infty},\\
\max_{0\leq l\leq k}\mbox{esssup}_{0\leq t\leq T}\|\partial^l_tf \|_X, &  {p=\infty}.
\end{array} \right.
\end{equation*}
 When $p=2$ or $k=0$, $W^{k,2}([0,T];X)$ is written as $H^k([0,T];X)$ or $L^p([0,T];X)$, respectively. Let $C([0,T]; X)$ denote the space of all continuous functions from $[0,T]$ to $X$ with the norm
$$\|f\|_{C([0,T];X)}=\max_{t\in [0,T]}\|f(t)\|_X.$$

Clearly, $C([0,T];X)$, $W^{k,p}([0,T];X)$ and $H^k([0,T];X)$ are all Banach spaces when $X$ is a Banach space.

Now we are in a position to discuss the physical problem mentioned above, namely, a viscoelastic body occupies the domain $\Omega\in\mathbb{R}^d$ with a Lipshitz continuous boundary $\Gamma$ and it will be deformed due to external forces. Especially, the deformation will be characterized by the viscoelastic property inside the body and the friction property at the boundary which could be divided into three disjoint measurable parts $\Gamma_1$, $\Gamma_2$ and $\Gamma_3$ with meas($\Gamma_1$)$>0$. Let $I:=[0,T]$ be the time interval and let $\varepsilon(u)$ and $\sigma(u)$ denote the linearized strain tensor and stress tensor of displacement vector $u$, respectively.  The relation between $\varepsilon(u)$ and $\sigma(u)$ is characterized by the viscoelastic rule, i.e.,  $\sigma = \mathbb{A}(\varepsilon(\dot{u}))+\mathbb{G}(\varepsilon(u))$, where $\mathbb{A}$, $\mathbb{G}$ are viscosity operator and elasticity operator, respectively.

Let $\nu$ denote the unit outer normal vector on $\Gamma$. The normal and tangential component of the displacement $u$ and stress field are denoted by
$$v_{\nu}=v\cdot \nu,\quad v_{\tau}=v-v_{\nu}v,\quad \sigma_{\nu}=(\sigma\nu)\cdot \nu, \quad \sigma_{\tau}=\sigma\nu-\sigma_{\nu}\nu.$$
Then we have the following Green formula:
\begin{equation}\label{equ:Greenformula}
(\sigma,\varepsilon(v))_Q+(Div \sigma,v)_H=\int_{\Gamma}\sigma \nu\cdot v d\Gamma,  \quad \forall v \in H.
\end{equation}
 We concern with the deformation field of the body on the time interval $[0,T]$ with $T>0$. The body is clamped on $\Gamma_1\times[0,T]$ so the displacement field vanishes there. A volume force $f_0$ acts in $\Omega \times [0,T]$ and surface traction of density $g$ acts on $\Gamma_2 \times [0,T]$. Affected by these forces, the body may contact with the obstacle, i.e.,  the foundation, on the contact surface $\Gamma_3$.
 Following the problem considered in \cite{chen 2001}, we can model the friction contact between the viscoelastic body and foundation with Coulomb's law of dry friction.  Assuming there is only sliding contact, we have
 $$
 |\sigma_{\tau}|=-\mu\sigma_{\nu},\quad \sigma_{\tau}=-\lambda(\dot{u}_{\tau}-v^*),\quad -\sigma_{\nu}=\beta |\dot{u}_{\nu}|,\quad \lambda \geq 0,
 $$
where $v^*$ is a constant vector which represents the displacement of the foundation. Let $w$ denote the wear function which measures the wear of the surface. Then it follows from \cite{chen 2001} that $u_{\nu}=-w$. Let $\rho$, $u_0$ and $v_0$ denote the mass density, initial displacement and initial velocity field, respectively. Under above assumptions, the formulation of the problem of viscoelastic contact with friction and wear reads as follows:
\begin{problem}\label{problem:P}
Find a deformation field $u:\Omega \times [0,T]\rightarrow \mathbb{R}^{d}$ and a stress field $\sigma :\Omega \times [0,T]\rightarrow \mathbb{S}^d$ such that
\begin{align*}
\begin{cases}
\sigma = \mathbb{A}(\varepsilon(\dot{u}))+\mathbb{G}(\varepsilon(u))\quad in\; \Omega \times (0,T),\\
Div \; \sigma + f_0 = \rho \ddot{u} \quad in \; \Omega \times (0,T),\\
u = 0 \quad on\; \Gamma_1 \times(0,T),\\
\sigma \nu =g \quad on \;\Gamma_2 \times (0,T),\\
-\sigma_{\nu}=\beta |\dot{u}_{\nu}|,\;|\sigma_{\tau}|=-\mu \sigma_{\nu},\;\sigma_{\tau}=-\lambda(\dot{u}_{\tau}-v^*),\lambda \geq 0\quad on \; \Gamma_3 \times (0,T),\\
u(0)=u_0,\; \dot{u}(0)=v_0 \quad in \; \Omega.
\end{cases}
\end{align*}
\end{problem}

To derive the variational formulation for Problem \ref{problem:P}, we need to introduce another space. Let $V$ denote a closed subspace of $H_1$ defined by
\begin{equation}\label{Vdefine}
V=\{v\in H_1| \mbox{$v=0$  on $\Gamma_1$}\}.
\end{equation}
We define the inner product $(\cdot,\cdot)_V$ on $V$ by setting
 $$(u,v)_V=(\varepsilon(u),\varepsilon(v))_Q, \quad \forall u,v\in V.$$
Since meas($\Gamma_1$)$>0$, Korn's inequality implies that there exists a positive number $C_K$ (depending only on $\Omega$) and $\Gamma_1$ such that
$$\|\varepsilon(v)\|_Q\geq C_K\|v\|_{H_1},\quad \forall v \in V.$$
Next we turn to derive the variational formulation for Problem \ref{problem:P}. From the second equality in Problem \ref{problem:P}, one has
\begin{equation}\label{equ:1weak}
(\rho \ddot{u}(t),v)_H-(Div\sigma(t),v)_H=(f_0(t),v), \quad \forall v \in V.
\end{equation}
Application of the Green formula \eqref{equ:Greenformula} enables us to rewrite \eqref{equ:1weak} as follows:
\begin{equation}\label{equ:2weak}
(\rho \ddot{u}(t),v)_H+(\sigma(t),\varepsilon(v))_Q=(f_0(t),v)_H+\int_{\Gamma}\sigma(t)\nu v d\Gamma,
\end{equation}
where
\begin{eqnarray}\label{equ:3weak}
\int_{\Gamma}\sigma(t)\nu vd\Gamma&=&\int_{\Gamma_1}\sigma(t)\nu vd\Gamma+\int_{\Gamma_2}\sigma(t)\nu vd\Gamma+\int_{\Gamma_3}\sigma(t)\nu vd\Gamma\nonumber\\
&=&\int_{\Gamma_2}g vd\Gamma+\int_{\Gamma_3}\sigma(t)\nu vd\Gamma\nonumber\\
&=&\int_{\Gamma_2}g vd\Gamma+\int_{\Gamma_3}(\sigma_{\nu}(t)\nu+\sigma_{\tau}(t)\ )vd\Gamma
\end{eqnarray}
due to $v=0$ on $\Gamma_1$ and $\sigma \nu=g$ on $\Gamma_2$.  Taking $v$ as $v-\dot{u}$ in \eqref{equ:2weak}, from the boundary condition of $\Gamma_3$, one has
\begin{eqnarray}\label{equ:4weak}
&&-\int_{\Gamma_3}(\sigma_{\nu}(t)\nu+\sigma_{\tau}(t))(v-\dot{u})d\Gamma \nonumber\\
&=&-\int_{\Gamma_3}\sigma_{\nu}(t) (v_{\nu}-\dot{u}_{\nu}(t))d\Gamma-\int_{\Gamma_3}\sigma_{\tau}(t) (v_{\tau}-\dot{u}_{\tau}(t))d\Gamma \nonumber\\
&=& \int_{\Gamma_3}\beta |\dot{u}_{\nu}| (v_{\nu}-\dot{u}_{\nu}(t))d\Gamma-\int_{\Gamma_3}\sigma_{\tau}(t) (v_{\tau}-v^*)d\Gamma
   +\int_{\Gamma_3}\sigma_{\tau}(t) (\dot{u}_{\tau}(t)-v^*)d\Gamma \nonumber\\
&\leq& \int_{\Gamma_3}\beta |\dot{u}_{\nu}| (v-\dot{u}(t))d\Gamma +\int_{\Gamma_3}|\sigma_{\tau}(t)| (|v_{\tau}-v^*|)d\Gamma
   -\int_{\Gamma_3}|\sigma_{\tau}(t)| |\dot{u}_{\tau}(t)-v^*|d\Gamma \nonumber\\
&=& \int_{\Gamma_3}\beta |\dot{u}_{\nu}| (v_{\nu}-\dot{u}_{\nu}(t))d\Gamma+\int_{\Gamma_3}\beta \mu|\dot{u}_{\nu}|(|v_{\tau}-v^*|-|\dot{u}_{\tau}(t)-v^*|)d\Gamma.
\end{eqnarray}
Letting
$$j(u,v)=\int_{\Gamma_3}\beta|u_{\nu}|(\mu|v_{\tau}-v^*|+v_{\nu})d\Gamma, \quad L(v)=(f_0(t),v)_H+\int_{\Gamma_2 }gvd\Gamma,$$
it follows from \eqref{equ:3weak} and \eqref{equ:4weak} that \eqref{equ:2weak} can  be rewritten as follows:
\begin{equation*}\label{equ:5weak}
\begin{aligned}
(\rho \ddot{u}(t),v-\dot{u}(t))_H+(\sigma(t),\varepsilon(v-\dot{u}(t)))_Q+j(\dot{u}(t),v)-j(\dot{u}(t),\dot{u}(t))\geq L(v-\dot{u}(t)).
\end{aligned}
\end{equation*}

Thus, the variational formulation for Problem \ref{problem:P} can be stated as follows:
\begin{problem}\label{problem:Pv}
Find a deformation field $u:[0,T]\rightarrow V$ such that
\begin{align*}
\begin{cases}
(\rho \ddot{u}(t),v-\dot{u}(t))_H+(\mathbb{A}(\varepsilon(\dot{u}(t)))+\mathbb{G}(\varepsilon(u(t))),\varepsilon(v-\dot{u}(t)))_Q+j(\dot{u}(t),v)-j(\dot{u}(t),\dot{u}(t))\\
\qquad  \geq L(v-\dot{u}(t)),\quad \forall v\in V, \; a.e. \; t \in [0,T], \\
u(0)=u_0,\;\dot{u}(0)=v_0.
\end{cases}
\end{align*}
\end{problem}

\section{Existence and uniqueness of Problem \ref{problem:Pv}}
\setcounter{equation}{0}

To solve Problem \ref{problem:Pv}, we consider a class of abstract variational inequalities in Banach spaces. Firstly, we introduce the Gelfand triple of Hilbert spaces which can be found in \cite{zengshengda 2017,Roubicek 2013}. Let $V$ be the strictly convex, reflexive and separable Banach space and $H$ the separable Hilbert space with $V$ embedding to $H$ densely and compactly. Identifying $H$ with its dual $H^*\cong H$, we obtain the Gelfand triple
$V\hookrightarrow H \hookrightarrow V^*$, where $\hookrightarrow$ denotes the continuous embedding. We write the embedding operator between $H$ and $V^*$ as $i^*$, which is continuous and compact. Therefore, we can consider the duality pairing
$\langle \cdot,\cdot\rangle_{V^*\times V}$ as continuous extension of the inner product $(\cdot,\cdot)_H$ on $H$, i.e.,
$$(u,v)_H=\langle u,v\rangle_{V^*\times V},\quad \forall u\in H,  \forall v\in V.$$

Let $\mathbb{V}\equiv L^2(I;V)$, $\mathbb{H}\equiv L^2(I;H)$ and $\mathbb{V}^*\equiv L^2(I;V^*)$. We define the duality pairing on $\mathbb{V}$ and the inner product on $\mathbb{H}$ as follows:
$$\langle S,L\rangle_{\mathbb{V}\times\mathbb{V}^*}=\int_0^T\langle S(t),L(t)\rangle_{V^*\times V}dt, \quad (S,L)_{\mathbb{H}}=\int_0^T(S(t),L(t))_Hdt.$$
It is well known that $\mathbb{H}$ is a Hilbert space and the space $\mathbb{W} \equiv \{v|v \in \mathbb{V},\dot{u} \in \mathbb{V}^*\}$ equipped with the graph norm $\|w\|_{\mathbb{W}}=\|w\|_{\mathbb{V}}+\|\dot{w}\|_{\mathbb{V}^*}$ is a separable reflexive Banach space, where $\dot{w}$ stands for the weak derivative of $w$. Furthermore, we know that $\mathbb{W} \hookrightarrow \mathbb{V}\hookrightarrow\mathbb{H}\hookrightarrow\mathbb{V}^*$ (see \cite{Roubicek 2013}).  We use the notation $\|\cdot\|$ to designate the norm in $V$ if there is no confusion.  Moreover, by no abuse of notation, we denote by $C$ a constant whose value may change from line to line when no confusing can arise.

Given two symmetry operators $A,\; B:V\to V^*$, a functional $j:V\times V \to \textrm{R}$ and two initial values $u_0\in V$ and $v_0 \in H$, we begin by study the following hyperbolic quasi-variational inequality:
\begin{problem}\label{problem:VI}
Find  $u \in\mathbb{V}$ with $\dot{u} \in \mathbb{W}$ such that
\begin{align*}
\begin{cases}
\langle\ddot{u}(t)+Au(t)+B\dot{u}(t)-f(t),v-\dot{u}(t)\rangle_{V^*\times V}+j(\dot{u}(t),v)-j(\dot{u}(t),\dot{u}(t))\geq 0,\;\forall v\in V, \;a.e. \; t \in I,\\
u(0)=u_0,\;\dot{u}(0)=v_0.
\end{cases}
\end{align*}
\end{problem}
In order to solve Problem \ref{problem:VI}, we impose the following assumptions.
\begin{itemize}\label{originalhypothesis}
\item[\textbf{H}(1)] There exists a constant $M_B>0$  such that
\begin{equation}\label{pro:pro1}
\langle B(u_1)-B(u_2),u_1-u_2\rangle\geq M_B\| u_1-u_2 \|^2_V, \quad  \forall u_1, u_2 \in V.
\end{equation}
\item[\textbf{H}(2)] There exists a constant $L_B>0$  such that
\begin{equation}\label{pro:pro2}
 \| B(u_1)-B(u_2) \|_{V^*}\leq L_B\| u_1-u_2 \|_V, \quad \forall u_1, u_2 \in V.
\end{equation}
\item[\textbf{H}(3)]$A\in \mathcal{L}(V,V^*)$ is strongly monotone, i.e., there exists a constant $M_A>0$ such that
\begin{equation}\label{pro:pro3}
\langle A(u),u\rangle\geq M_A\| u \|^2_V, \quad \forall u\in V.
\end{equation}
\item[\textbf{H}(4)] The norm of $A$ is $L_A$, i.e.
\begin{equation}\label{pro:pro4}
 \| Au \|_{V^*}\leq L_A\| u \|_V, \quad \forall u_1, u_2 \in V.
\end{equation}
\item[\textbf{H}(5)] For any $u,v \in V$,
\begin{equation}\label{pro:pro5}
\langle Au,v\rangle=\langle Av,u\rangle, \quad
\end{equation}
\item[\textbf{H}(6)]
There exists a constant $L_j>0$  such that
\begin{equation}\label{pro:pro6}
j(g_1,v_2)+j(g_2,v_1)-j(g_1,v_1)-j(g_2,v_2)\leq L_j\|g_1-g_2\|_V\|v_1-v_2\|_V, \quad \forall g_1,g_2,v_1, v_2 \in V.
\end{equation}
\item[\textbf{H}(7)]
There exists a constant $C_j>0$  such that
\begin{equation}\label{pro:pro7}
j(g,v_1)-j(g,v_2)\leq C_j\|g\|_V\|v_1-v_2\|_V, \quad j(v_1,g)-j(v_2,g)\leq C_j\|g\|_V\|v_1-v_2\|_V, \quad \forall g,v_1, v_2 \in V.
\end{equation}
\item[\textbf{H}(8)] The function $f$ satisfies
\begin{equation}\label{pro:pro8}
f \in H^2(I;V^*).
\end{equation}
\item[\textbf{H}(9)]  For any $u\in V$,
\begin{equation}\label{pro:pro9}
j(u,\cdot) \; \mbox{is a convex functional in V for all u }\in V.
\end{equation}
\end{itemize}
\begin{remark}
We would like to mention the following facts: (a) the assumption \eqref{pro:pro6} was first introduced by Han et al. \cite{Han 2001} to study a class of parabolic quasi-variational inequalities; (b) the assumptions \eqref{pro:pro7} and \eqref{pro:pro9} imply that $j(u,\cdot)$ is a continuous convex functional in $V$ for all $u \in V$.
\end{remark}
\begin{remark}\label{remark:r1} We would like to point out the following facts: (a) If $B=0$ and $j$ is proper, convex and lower semi-continuous, then Problem \ref{problem:VI} degenerates to the following problem (\cite{zengshengda 2017}):
\begin{align*}
\begin{cases}
\langle\ddot{u}(t)+ Au(t)-f(t),v-\dot{u}(t)\rangle+j(v)-j(\dot{u}(t))\geq 0,\quad \forall v\in V, \; a.e\; t \in I,\\
u(0)=u_0,\;\dot{u}(0)=v_0;\quad u_0\in \mathbb{W},v_0\in \mathbb{V};
\end{cases}
\end{align*}
(b) If $\ddot{u}(t)$ can be neglected, then Problem \ref{problem:VI} reduces to the parabolic variational inequality studied by Han and Sofonea \cite{HanSonfonea 2000}; (c) If $j$ contains only one variable, then Problem \ref{problem:VI} is the classic hyperbolic variational inequality considered in \cite{Atkinson 2017}.
\end{remark}

Similar to the study of \cite{zengshengda 2017}, we will employ the Rothe method to prove the existence and uniqueness of the solution for Problem \ref{problem:VI}. It adopts the time semi-discrete scheme to obtain a sequence of convergence function. At each time step, we need to solve an elliptic variational problem. Then, we apply the piecewise constant and piecewise affine interpolation to approximate the solution of Problem \ref{problem:VI}.

For any given $N \in \mathbb{N}$, let $f_{\tau}^k=\frac{1}{\tau}\int_{t_{k-1}}^{t_k}f(t)dt$ for $k = 1, \cdots, N$, where $\tau = \frac{T}{N}$ and $t_k = k\tau$. We use the following notations for simplicity
\begin{align}\label{Rothe}
\nu_{\tau}^k=\frac{u_{\tau}^{k}-u_{\tau}^{k-1}}{\tau},\quad z_{\tau}^k=\frac{\nu_{\tau}^{k}-\nu_{\tau}^{k-1}}{\tau},\quad
u_{\tau}^k=u_0+\tau \sum_{i=1}^{k} \nu_{\tau}^i,\quad u_{\tau}^0=u_0,\quad \nu_{\tau}^0=v_0.
\end{align}
We now consider the following semi-discrete scheme for solving Problem \ref{problem:VI}.
\begin{problem}\label{problem:Rothe}
 Find $\{\nu_{\tau}^k\}_{k=1}^{N}\subset V$ with $u_{\tau}^0=u_0$ and $\nu_{\tau}^0=v_0$ such that
\begin{align}\label{equ:Rothe}
\begin{cases}
\left(\frac{\nu_{\tau}^k-\nu_{\tau}^{k-1}}{\tau},v-\nu_{\tau}^k\right)_H
+\left\langle A(u_0+\tau\sum_{i=1}^{k}\nu_{\tau}^i)+B\nu_{\tau}^k-f_{\tau}^k,v-\nu_{\tau}^k\right\rangle
+j(\nu_{\tau}^k,v)-j(\nu_{\tau}^k,\nu_{\tau}^k)\geq 0,\\ \quad \forall v\in V, \; k=1,2,\cdots,N.\\
\end{cases}
\end{align}
\end{problem}

\begin{lemma} \label{lem:Grownwell}\cite{Roubicek 2013}
Provided $e_n \leq cg_n+\tau \sum_{k=1}^{n-1}e_k$ with $n=1,2,\dots,N$, one has
$$\max_{1\le k\le N} e_k \leq C \max_{1\le k\le N} g_k,$$
and
$$\max_{1\le k\le N} e_k \leq\max_{1\le k\le N} \hat{c} \left( g_k+\tau\sum_{1\le j\le k}g_j\right),$$
where $\hat{c}$, $c$, $C$, $\tau$, $\{e_k\}_{k=1}^N$ and $\{g_k\}_{k=1}^N$ are all positive numbers.
\end{lemma}

\begin{lemma}\label{lem:Rothe}
Assume that conditions \eqref{pro:pro1}-\eqref{pro:pro4}, \eqref{pro:pro6} and \eqref{pro:pro9} are satisfied with $M_B>L_j$. Then there exists a constant $\tau_0$ such that, for any $\tau \in (0,\tau_0)$, Problem \ref{problem:Rothe} has a unique solution.
\end{lemma}
\begin{proof}
Suppose that $\{\nu_{\tau}^i\}_{i=1}^{k-1}$ are given with $k\geq2$. Then we can rewrite Problem \ref{problem:Rothe} as follows:  Find $\nu_{\tau}^k\in V$ such that
\begin{eqnarray}\label{equ:variRothe1}
&&\left(\frac{\nu_{\tau}^k}{\tau},v-\nu_{\tau}^k\right)_H
+\langle B\nu_{\tau}^k+\tau A\nu_{\tau}^k,v-\nu_{\tau}^k\rangle
+j(\nu_{\tau}^k,v)-j(\nu_{\tau}^k,\nu_{\tau}^k)\nonumber\\
&\geq& \left\langle f_{\tau}^k-A\left(u_0+\tau \sum_{i=1}^{k-1}\nu_{\tau}^i\right),v-\nu_{\tau}^k\right\rangle
+\left(\frac{\nu_{\tau}^{k-1}}{\tau},v-\nu_{\tau}^k\right)_H.
\end{eqnarray}
Let
$$
F_{\tau}=f_{\tau}^k+i^*i\frac{\nu_{\tau}^{k-1}}{\tau}-A\left(u_0+\tau \sum_{i=1}^{k-1}\nu_{\tau}^i\right).
$$
Then Problem \ref{problem:Rothe} can be transformed as follows: Find $u \in V$ such that
\begin{eqnarray*}\label{equ:variRothe2}
\left\langle  i^*i\frac{u}{\tau}+Bu+\tau Au,v-u\right\rangle
+j(u,v)-j(u,u)\geq \langle F_{\tau},v-u\rangle, \quad \forall v\in V.
\end{eqnarray*}
For any given $g,\;h\in V$, we consider the following variational inequality of finding $u_{gh}\in V$ such that
\begin{eqnarray}\label{equ:variRothe3}
\left\langle i^*i\frac{u_{gh}}{\tau}+Bu_{gh}+ \tau Ah,v-u_{gh}\right\rangle+j(g,v)-j(g,u)
\geq \langle F_{\tau},v-u_{gh}\rangle, \quad  \forall v\in V.
\end{eqnarray}
We note that problem \eqref{equ:variRothe3} is an elliptic variational inequality on the Banach space $V$. In terms of \eqref{pro:pro1}-\eqref{pro:pro4}, \eqref{pro:pro9} and the definition of embedding operator $i$,  we can see that the classical elliptic variational inequality \eqref {equ:variRothe3} has a unique solution $u_{gh}$ (see \cite{Glowinski 1981,Xiao 2012}). Moreover, by Lemma 3.1 of Xiao et al. \cite{Xiao 2012}, we have $u_{gh} \in V$.

Next, we show that the mapping $g\mapsto u_{gh}$ is contractive for any given $h\in V$. In fact, letting $g=g_1$ and $g=g_2$ in \eqref{equ:variRothe3}, respectively, one has
\begin{eqnarray}\label{11}
\left\langle i^*i\frac{u_{g_{1h}}}{\tau}+Bu_{g_{1h}}+ \tau Ah,v-u_{g_{1h}}\right\rangle+j(g_{1},v)-j(g_{1},u_{g_{1h}})
\geq \langle F_{\tau},v-u_{g_{1h}}\rangle, \quad  \forall v\in V
\end{eqnarray}
and
\begin{eqnarray}\label{12}
\left\langle i^*i\frac{u_{g_{2h}}}{\tau}+Bu_{g_{2h}}+ \tau Ah,v-u_{g_{2h}}\right\rangle+j(g_{2},v)-j(g_{2},u_{g_{2h}})
\geq \langle F_{\tau},v-u_{g_{2h}}\rangle, \quad  \forall v\in V.
\end{eqnarray}
Taking $v=u_{g_{2h}}$ and $v=u_{g_{1h}}$ in \eqref{11} and \eqref{12}, respectively, and adding the above two inequalities, we have
\begin{eqnarray}\label{equ:variRothe4}
&&\left(\frac{u_{g_{1h}}-u_{g_{2h}}}{\tau},u_{g_{1h}}-u_{g_{2h}}\right)_H+\langle Bu_{g_{1h}}-Bu_{g_{1h}},u_{g_{1h}}-u_{g_2h}\rangle\nonumber\\
&\leq& j(g_1,u_{g_{2h}})+j(g_2,u_{g_{1h}})-j(g_1,u_{g_{1h}})-j(g_2,u_{g_{2h}}).
\end{eqnarray}
Employing the assumptions of operator $B$ and functional $j$, it follows from \eqref{equ:variRothe4} that
\begin{eqnarray*}\label{equ:variRothe5}
M_B\left\|u_{g_1h}-u_{g_2h}\right\|\leq L_j\|g_1-g_2\|.
\end{eqnarray*}
Since $M_B> L_j$, we conclude that $g\mapsto u_{gh}$ is a contractive mapping and so it has a unique fixed point $u_h\in V$ by Banach fixed point theorem. This implies  that $u_h$ satisfies the following variational inequality:
\begin{eqnarray}\label{equ:variRothe6}
\left(\frac{u_h}{\tau},v-u_h\right)_H
+\langle Bu_h+ \tau Ah ,v-u_h\rangle+j(u_h,v)-j(u_h,u_h)\geq \langle F_{\tau},v-u_h\rangle, \quad \forall v \in V.
\end{eqnarray}

Similarly, taking $h=h_1$ and $h=h_2 \in V$ in \eqref{equ:variRothe6}, respectively, one has
\begin{eqnarray*}\label{equ:variRothe7}
(M_B-L_j)\|u_{h1}-u_{h2}\|\leq \tau_0 L_A\|h_1-h_2\|.
\end{eqnarray*}
Choosing $\tau_0$ such that $\frac{\tau_0 L_A}{M_B-L_j}<1$, we can derive that the mapping $h\mapsto u_h$ has a unique fixed point $u\in V$ and so $u$ is a unique solution of the variational inequality \eqref{equ:variRothe1}.

Finally, for $k=1$, Problem \ref{problem:Rothe} can be stated as follows: Find $\nu_{\tau}^k\in V$ such that, for all $v \in V$,
\begin{eqnarray}\label{equ:variRothe9}
\left(\frac{\nu_{\tau}^1}{\tau},v-\nu_{\tau}^1\right)_H
+\langle B\nu_{\tau}^1+\tau A\nu_{\tau}^1,v-\nu_{\tau}^1\rangle
+j(\nu_{\tau}^1,v)-j(\nu_{\tau}^1,\nu_{\tau}^1)\geq\left \langle f_{\tau}^1+\frac{i^*i\nu_{\tau}^0}{\tau}-Au_0,v-\nu_{\tau}^1\right\rangle
\end{eqnarray}
Similar to the case that $k\ge 2$, it can be easily seen that the variational inequality \eqref{equ:variRothe9} has a unique solution provided $M_B> L_j$.  This ends the proof.
\end{proof}

\begin{lemma}\label{lem:usd}
Assume that conditions \eqref{pro:pro1}-\eqref{pro:pro8} are satisfied with $\tau\in(0,\tau_0)$ and $M_B<2C_j$. Then there exists a constant $C^*$ being independent of $\tau$ such that
$$\|u_\tau^1\|\leq C^*,\quad \|\nu_\tau^1\|\leq C^*, \quad \|z_\tau^1\|_H\leq C^*.$$
\end{lemma}
\begin{proof}
Taking $v=v_0$ in \eqref{equ:variRothe9}, one has
\begin{eqnarray*}\label{equ:ie1}
\left(\frac{\nu_{\tau}^1-v_0}{\tau},v_0-\nu_{\tau}^1\right)_H
+\langle B\nu_{\tau}^1+\tau A\nu_{\tau}^1,v_0-\nu_{\tau}^1\rangle
+j(\nu_{\tau}^1,v_0)-j(\nu_{\tau}^1,\nu_{\tau}^1)
\geq \langle f_{\tau}^1-Au_0,v_0-\nu_{\tau}^1\rangle.
\end{eqnarray*}
It follows from \eqref{pro:pro1}-\eqref{pro:pro7} and \eqref{pro:pro8} that
\begin{eqnarray*}
C_j\|\nu_{\tau}^1\|\|\nu_{\tau}^1-v_0\|
&\ge& \frac{1}{\tau}\left\|v_0-\nu_{\tau}^{1}\right\|_H^2+
\left\langle B\nu^1_{\tau},\nu^1_{\tau}-v_0\right\rangle\nonumber\\
&&\mbox{} +\frac{\tau}{2}\left\langle A\nu^1_{\tau},\nu^1_{\tau}\right\rangle-\frac{\tau}{2}\left\langle Av_0,v_0\right\rangle
+\frac{\tau}{2}\left\langle A\nu^1_{\tau}-Av_0,\nu^1_{\tau}-v_0\right\rangle\nonumber\\
&&\mbox{} +\langle Au_0-f^1_{\tau},\nu^1_{\tau}-v_0\rangle.
\end{eqnarray*}
This shows that
\begin{eqnarray*}\label{equ:ie3}
&& C_j\|\nu_{\tau}^1\|\|\nu_{\tau}^1-v_0\|+\frac{\tau L_A}{2}\left\|v_0\right\|^2
+\|Au_0-f^1_{\tau}\|\|\nu^1_{\tau}-v_0\|\nonumber\\
&\geq& \frac{1}{\tau}\|v_0-\nu_{\tau}^{1}\|_H^2+\frac{M_B}{2}\|\nu^1_{\tau}\|\|v_0-\nu_{\tau}^{1}\|
+\frac{\tau M_A}{2}\left\|\nu^1_{\tau}\right\|^2
+\frac{\tau M_A}{2}\|\nu^1_{\tau}-v_0\|^2
\end{eqnarray*}
and so
\begin{eqnarray*}
&& C_j\|\nu_{\tau}^1\|\|z_{\tau}^1\|+\frac{ L_A}{2}\left\|v_0\right\|^2
+\|Au_0-f^1_{\tau}\|\|z^1_{\tau}\|\nonumber\\
&\geq& \|z_{\tau}^{1}\|_H^2+\frac{M_B}{2}\|\nu^1_{\tau}\|\|z_{\tau}^{1}\|
+\frac{ M_A}{2}\left\|\nu^1_{\tau}\right\|^2
+\frac{ M_A}{2}\|\nu^1_{\tau}-v_0\|^2\\
&\geq& \|z_{\tau}^{1}\|_H^2+\frac{M_B}{2}\|\nu^1_{\tau}\|\|z_{\tau}^{1}\|
+\frac{ M_A}{2}\left\|\nu^1_{\tau}\right\|^2
+\frac{ M_A}{2\tau^2_0}\|z_\tau^1\|^2.
\end{eqnarray*}
This yields
\begin{eqnarray*}\label{equ:ie4}
&&\frac{C_j-\frac{M_B}{2}}{4\epsilon_1}\|\nu_{\tau}^1\|^2+\epsilon_1\left(C_j-\frac{M_B}{2}\right)\|z_{\tau}^1\|^2
+\frac{1}{4\epsilon_2}\|Au_0-f^1_{\tau}\|^2
+\epsilon_2\|z_{\tau}^1\|^2+\frac{ L_A}{2}\|v_0\|^2\nonumber \\
 &\geq&\|z_{\tau}^{1}\|_H^2+\frac{ M_A}{2}\left\|\nu^1_{\tau}\right\|^2
+\frac{ M_A}{2\tau^2_0}\|z^1_{\tau}\|^2,
\end{eqnarray*}
where $\epsilon_1$, $\epsilon_2$ are positive numbers. Thus, one has
\begin{eqnarray}\label{equ:ie5}
&&\left(\frac{M_A}{2}-\frac{C_j-\frac{M_B}{2}}{4\epsilon_1}\right)\|v_{\tau}^1\|^2
+ \|z_\tau^1\|_H^2\nonumber\\
&\leq& \frac{1}{4\epsilon_2}\|Au_0-f_\tau^1\|^2+\frac{L_A}{2}\|v_0\|^2
+\left(\epsilon_1\left(C_j-\frac{M_B}{2}\right)+\epsilon_2-\frac{ M_A}{2\tau^2_0}\right)\|z_{\tau}^1\|^2.
\end{eqnarray}
Since
\begin{eqnarray}\label{equ:ie6}
\|f_\tau^1\|
&=&\frac{1}{\tau}\left\|\int_0^\tau f(s) -f(0)ds+f(0)\right\|\nonumber\\
&\leq&\frac{1}{\tau}\int_0^\tau\int_0^s\left\| \frac{df}{du}\right\| duds+\|f(0)\|_{V^*}\nonumber\\
&\leq&\frac{1}{\tau}\int_0^\tau\left(\int_0^s\left\| \frac{df}{du}\right\|^2 du\right)^{\frac{1}{2}}\sqrt{s}ds+\|f(0)\|_{V^*}\nonumber\\
&\leq&\sqrt{\tau_0} \left\| \frac{df}{du}\right\|_{\mathbb{V^*}}+\|f(0)\|_{V^*},
\end{eqnarray}
it follows that $\|Au_0-f_\tau^1\|$ is bounded, i.e.,  there exists a constant $M$ being independent of $\tau$ such that $\|Au_0-f_\tau^1\|\le M$.
By the fact that $M_B<2C_j$, we can choose positive numbers
$$\epsilon_1>\frac{2C_j-M_B}{4M_A},\quad \epsilon_2=\left(\frac{ M_A}{2\tau^2_0}-\epsilon_1\left(C_j-\frac{M_B}{2}\right)\right)>0$$
such that
\begin{align}\label{equ:ie7}
\epsilon_1\left(C_j-\frac{M_B}{2}\right)+\epsilon_2-\frac{ M_A}{2\tau^2_0}=0,\quad
\frac{M_A}{2}-\frac{C_j-\frac{M_B}{2}}{4\epsilon_1}>0.
\end{align}
Thus, the following inequalities can be obtained from \eqref{equ:ie5}:
\begin{eqnarray*}\label{equ:ie8}
\|\nu_{\tau}^{1}\|\leq C^*, \quad
\|z_\tau^1\|_H\leq C^*,
\end{eqnarray*}
where $C^*$ is a constant. Since $u_\tau^1=u_0+\tau \nu_\tau^1$, we have
\begin{eqnarray*}\label{equ:ie9}
\|u_\tau^1\|\leq\tau_0\|\nu_\tau^1\|+\|u_0\|.
\end{eqnarray*}
This completes the proof.
\end{proof}

\begin{lemma}\label{lem:Rothe2}
Under assumptions \eqref{pro:pro1}-\eqref{pro:pro8} and $M_B>L_j$,  there exist two constants $\tau_0> 0$ and  $C > 0$ being independent of $\tau$ such that
\begin{eqnarray*}\label{lemequ:Rothe30}
\max_{2\le k\le N}\|z_{\tau}^k\|_H\leq C,\quad \max_{2\le k\le N}\|\nu_{\tau}^k\|\leq C, \quad \sum_{k=2}^{N}\|\nu_{\tau}^k-\nu_{\tau}^{k-1}\|\leq C, \quad \sum_{k=2}^{N}\|z_{\tau}^k-z_{\tau}^{k-1}\|_H\leq C,\quad
\tau \in (0, \tau_0).
\end{eqnarray*}
\end{lemma}
\begin{proof}
For $k\geq2$, similar to the proof of Lemma \ref{lem:Rothe}, we consider the $k$-th inequality in \eqref{problem:Rothe} with $v=\nu_{\tau}^{k-1}$ and $(k-1)$-th inequality in \eqref{problem:Rothe} with $v=\nu_{\tau}^k$. Then we add these two inequalities and get
\begin{eqnarray*}\label{lemequ:Rothe21}
&&\frac{1}{\tau}(\nu_{\tau}^k+\nu_{\tau}^{k-2}-2\nu_{\tau}^{k-1},\nu_{\tau}^k-\nu_{\tau}^{k-1})_H
+\tau\langle A\nu_{\tau}^{k},\nu_{\tau}^k-\nu_{\tau}^{k-1}\rangle\nonumber\\
&& \mbox{}+\langle B\nu_{\tau}^k-B\nu_{\tau}^{k-1},\nu_{\tau}^k-\nu_{\tau}^{k-1}\rangle
-\langle f_{\tau}^k-f_{\tau}^{k-1},\nu_{\tau}^k-\nu_{\tau}^{k-1}\rangle\nonumber\\
&\leq& j(\nu_{\tau}^{k},\nu_{\tau}^{k-1})+ j(\nu_{\tau}^{k-1},\nu_{\tau}^{k})- j(\nu_{\tau}^{k},\nu_{\tau}^{k})- j(\nu_{\tau}^{k-1},\nu_{\tau}^{k-1}).
\end{eqnarray*}
By \eqref{pro:pro6} and \eqref{Rothe}, we have
\begin{eqnarray}\label{lemequ:Rothe22}
&&(z_{\tau}^k-z_{\tau}^{k-1},z_{\tau}^k)_H
+\langle A\nu_{\tau}^{k},\nu_{\tau}^k-\nu_{\tau}^{k-1}\rangle
+\frac{1}{\tau}\langle B\nu_{\tau}^k-B\nu_{\tau}^{k-1},\nu_{\tau}^k-\nu_{\tau}^{k-1}\rangle\nonumber\\
&\leq&\left\langle\frac{f_{\tau}^k-f_{\tau}^{k-1}}{\tau},\nu_{\tau}^k-\nu_{\tau}^{k-1}\right\rangle
+\frac{L_j}{\tau}\|\nu_{\tau}^{k}-\nu_{\tau}^{k-1}\|^2.
\end{eqnarray}
It follows from \eqref{pro:pro5} and the symmetry of inner product that
\begin{align}\label{Rothe23}
\begin{cases}
\langle A\nu_{\tau}^{k},\nu_{\tau}^k-\nu_{\tau}^{k-1}\rangle
=\frac{1}{2}\langle A\nu_{\tau}^{k},\nu_{\tau}^k\rangle-\frac{1}{2}\langle A\nu_{\tau}^{k-1},\nu_{\tau}^{k-1}\rangle
+\frac{1}{2}\langle A\nu_{\tau}^{k-1}-A\nu_{\tau}^{k},\nu_{\tau}^{k-1}-\nu_{\tau}^{k}\rangle\\
(z_{\tau}^k-z_{\tau}^{k-1},z_{\tau}^k)_H=\frac{1}{2}\left(\|z_{\tau}^k\|_H^2-\|z_{\tau}^{k-1}\|_H^2+\|z_{\tau}^k-z_{\tau}^{k-1}\|_H^2\right).
\end{cases}
\end{align}
Then \eqref{lemequ:Rothe22} can be rewritten as follows:
\begin{eqnarray*}\label{lemequ:Rothe24}
&&\frac{1}{2}\left(\|z_{\tau}^k\|_H^2-\|z_{\tau}^{k-1}\|_H^2+\|z_{\tau}^k-z_{\tau}^{k-1}\|_H^2\right)
+\frac{1}{2}\langle A\nu_{\tau}^{k},\nu_{\tau}^k\rangle-\frac{1}{2}\langle A\nu_{\tau}^{k-1},\nu_{\tau}^{k-1}\rangle\nonumber\\
&&\mbox{}+\frac{1}{2}\langle A\nu_{\tau}^{k-1}-A\nu_{\tau}^{k},\nu_{\tau}^{k-1}-\nu_{\tau}^{k}\rangle
+\frac{1}{\tau}\langle B\nu_{\tau}^k-B\nu_{\tau}^{k-1},\nu_{\tau}^k-\nu_{\tau}^{k-1}\rangle\nonumber\\
&\leq& \left\langle\frac{f_{\tau}^k-f_{\tau}^{k-1}}{\tau},\nu_{\tau}^k-\nu_{\tau}^{k-1}\right\rangle
+\frac{L_j}{\tau}\|\nu_{\tau}^{k}-\nu_{\tau}^{k-1}\|^2.
\end{eqnarray*}
Summing up the above inequalities from $2$ to $n$ with $2<n<N$, we have
\begin{eqnarray}\label{lemequ:Rothe25}
&&\frac{1}{2}\langle A\nu_{\tau}^{n},\nu_{\tau}^n\rangle
-\frac{1}{2}\langle A\nu_{\tau}^{1},\nu_{\tau}^{1}\rangle
+\frac{1}{2}\sum_{k=2}^{n}\langle A\nu_{\tau}^{k-1}-A\nu_{\tau}^{k},\nu_{\tau}^{k-1}-\nu_{\tau}^{k}\rangle\nonumber\\
&&\mbox{}+\frac{1}{\tau}\sum_{k=2}^{n}\langle B\nu_{\tau}^k-B\nu_{\tau}^{k-1},\nu_{\tau}^k-\nu_{\tau}^{k-1}\rangle
+\frac{1}{2}\left(\|z_{\tau}^n\|_H^2-\|z_{\tau}^1\|_H^2+\sum_{k=2}^{n}\|z_{\tau}^k-z_{\tau}^{k-1}\|_H^2\right)\nonumber\\
&\leq& \sum_{k=2}^{n}\left\langle \frac{f_{\tau}^k-f_{\tau}^{k-1}}{\tau},\nu_{\tau}^k-\nu_{\tau}^{k-1}\right\rangle
+\sum_{k=2}^{n}\frac{L_j}{\tau}\|\nu_{\tau}^{k}-\nu_{\tau}^{k-1}\|^2.
\end{eqnarray}
Since $f(0)\in V^*$, we can extend $f$ to the interval $(-\tau,T]$ as $f(t)=f(0)$ while $t\in (-\tau,0]$. It follows that $f\in H^2(-\tau,T;V^*)$. Let $F_\tau^k=\frac{f_{\tau}^k-f_{\tau}^{k-1}}{\tau}$. Then the right term of \eqref{lemequ:Rothe25} can be calculated as follows:
\begin{eqnarray}\label{rightterm}
\sum_{k=2}^{n}\left\langle\frac{f_{\tau}^k-f_{\tau}^{k-1}}{\tau},\nu_{\tau}^k-\nu_{\tau}^{k-1}\right\rangle
&=&\sum_{k=2}^{n}\langle F_\tau^k,\nu_{\tau}^k-\nu_{\tau}^{k-1}\rangle=\sum_{k=2}^{n}\langle F_\tau^k,\nu_{\tau}^k\rangle-\sum_{k=2}^{n}\langle F_\tau^k,\nu_{\tau}^{k-1}\rangle\nonumber\\
&=&\sum_{k=2}^{n}\langle F_\tau^k,\nu_{\tau}^k\rangle-\sum_{k=2}^{n}\langle F_\tau^k-F_\tau^{k-1},\nu_{\tau}^{k-1}\rangle
-\sum_{k=2}^{n}\langle F_\tau^{k-1},\nu_{\tau}^{k-1}\rangle\nonumber\\
&=& \langle F_\tau^n,\nu_{\tau}^n\rangle-\langle F_\tau^1,\nu_{\tau}^1\rangle
-\sum_{k=2}^{n}\tau\left\langle \frac{F_\tau^k-F_\tau^{k-1}}{\tau},\nu_{\tau}^{k-1}\right\rangle.
\end{eqnarray}
We note that
\begin{eqnarray*}
\left\|\frac{F_{\tau}^k-F_{\tau}^{k-1}}{\tau}\right\|_{V^*}
&=&\frac{1}{\tau^2}\|f_\tau^k-2f_\tau^{k-1}+f_\tau^{k-1}\|_{V^*}\nonumber\\
&=&\frac{1}{\tau^3}\left\|\int_{t_{k-1}}^{t_k}\int_{s_1-\tau}^{s_1}\int_{s_2-\tau}^{s_2}\frac{d^2f}{{ds}^2}dsds_2ds_1\right\|_{V^*}\nonumber\\
&\leq&\frac{1}{\tau^3}\int_{t_{k-1}}^{t_k}\int_{s_1-\tau}^{s_1}\int_{s_2-\tau}^{s_2}\left\|\frac{d^2f}{{ds}^2}\right\|_{V^*}dsds_2ds_1\nonumber\\
&\leq&\frac{1}{\tau^3}\int_{t_{k-1}}^{t_k}\int_{s_1-\tau}^{s_1}\int_{t_k-2\tau}^{t_k}\left\|\frac{d^2f}{{ds}^2}\right\|_{V^*}dsds_2ds_1\nonumber\\
&\leq&\frac{1}{\tau^3}\int_{t_{k-1}}^{t_k}\int_{t_k-\tau}^{t_k}\int_{t_k-3\tau}^{t_k}\left\|\frac{d^2f}{{ds}^2}\right\|_{V^*}dsds_2ds_1\nonumber\\
&\leq&\frac{1}{\tau^3}\int_{t_{k-1}}^{t_k}\int_{t_k-\tau}^{t_k}\left(\int_{t_k-3\tau}^{t_k}\left\|\frac{d^2f}{{ds}^2}\right\|_{V^*}^2ds\right)^{\frac{1}{2}}\sqrt{3\tau}ds_2ds_1\nonumber\\
&\leq&\frac{2\sqrt{3}}{\sqrt{\tau}}\left(\int_{t_k-3\tau}^{t_k}\left\|\frac{d^2f}{{ds}^2}\right\|_{V^*}^2ds\right)^\frac{1}{2}.
\end{eqnarray*}
Thus, the right term of \eqref{rightterm} has the following estimate:
\begin{eqnarray}\label{aa}
\tau\sum_{k=2}^{n}\left|\left\langle \frac{F_\tau^k-F_\tau^{k-1}}{\tau},\nu_{\tau}^{k-1}\right\rangle\right|
&\leq&\tau\sum_{k=2}^{n}\frac{1}{4\epsilon}\left\|\frac{F_\tau^k-F_\tau^{k-1}}{\tau}\right\|_{V^*}^2
+\tau\sum_{k=2}^{n}\epsilon\|\nu_{\tau}^{k-1}\|^2\nonumber\\
&\leq& \tau\sum_{k=2}^{n}\frac{3}{\epsilon \tau}\int_{t_k-3\tau}^{t_k}
\left\|\frac{d^2f}{{ds}^2}\right\|_{V^*}^2ds+\tau\sum_{k=2}^{n}\epsilon\left\|\nu_{\tau}^{k-1}\right\|^2\nonumber\\
&\leq&\frac{9}{\epsilon}\sum_{k=0}^{n}\int_{t_{k-1}}^{t_k}\left\|\frac{d^2f}{{ds}^2}\right\|_{V^*}^2ds
+\tau\sum_{k=2}^{n}\epsilon\|\nu_{\tau}^{k-1}\|^2\nonumber\\
&\leq&\frac{9}{\epsilon} \left\|\frac{d^2f}{{ds}^2}\right\|_{H(-\tau,T;V^*)}^2
+\tau\epsilon\sum_{k=2}^{n}\|\nu_{\tau}^{k-1}\|^2,
\end{eqnarray}
where $\epsilon$ is an arbitrary positive number. It follows from \eqref{rightterm}, \eqref{aa} that \eqref{lemequ:Rothe25} can be converted as follows:
\begin{eqnarray*}\label{lemequ:Rothe26}
&&\frac{M_A}{2}\|\nu_{\tau}^{n}\|^2-\frac{L_A}{2}\|\nu_\tau^{1}\|^2+\frac{M_A}{2}\sum_{k=2}^{n}\|\nu_{\tau}^{k-1}-\nu_{\tau}^{k}\|^2\nonumber\\
&&\mbox{}+\frac{1}{2}\left(\|z_{\tau}^n\|_H^2-\|z_{\tau}^1\|_H^2+\sum_{k=2}^{n}\|z_{\tau}^k-z_{\tau}^{k-1}\|_H^2\right)
+\sum_{k=2}^{n}\frac{(M_B-L_j)}{\tau}\|\nu_{\tau}^k-\nu_{\tau}^{k-1}\|^2\nonumber\\
&\leq&\frac{9}{\epsilon} \left\|\frac{d^2f}{{ds}^2}\right\|_{H(-\tau,T;V^*)}^2
+\tau\epsilon\sum_{k=2}^{n}\|\nu_{\tau}^{k-1}\|^2+|\langle F_\tau^n,\nu_{\tau}^n\rangle|
+|\langle F_\tau^1,\nu_{\tau}^1\rangle|\nonumber\\
&\leq& \frac{9}{\epsilon} \left\|\frac{d^2f}{{ds}^2}\right\|_{H(-\tau,T;V^*)}^2
+\tau\epsilon\sum_{k=2}^{n}\|\nu_{\tau}^{k-1}\|^2
+\frac{1}{4\eta}\|F_\tau^n\|_{V^*}^2+\eta\|\nu_{\tau}^n\|^2+|\langle F_\tau^1,\nu_{\tau}^1\rangle|,
\end{eqnarray*}
where $\eta$ is an arbitrary positive number and $\eta<\frac{M_A}{2}$. Thus,
\begin{eqnarray}\label{lemequ:Rothe27}
&&\left(\frac{M_A}{2}-\eta\right)\|\nu_{\tau}^{n}\|^2+\frac{1}{2}\|z_{\tau}^n\|_H^2+\frac{1}{2}\sum_{k=2}^{n}\|z_{\tau}^k-z_{\tau}^{k-1}\|_H^2
+\left(\frac{M_B-L_j}{\tau_0}+\frac{M_A}{2}\right)\sum_{k=2}^{n}\|\nu_{\tau}^{k-1}-\nu_{\tau}^{k}\|^2\nonumber\\
&\leq&\frac{1}{2}\|z_\tau^1\|_H^2+\frac{L_A}{2}\|\nu_\tau^1\|^2
+\frac{9}{\epsilon}\left \|\frac{d^2f}{{ds}^2}\right\|_{H(-\tau,T;V^*)}^2
+\tau\epsilon\sum_{k=1}^{n-1}\|\nu_{\tau}^{k}\|^2\nonumber\\
&& \mbox{} +\frac{1}{4\eta}\|F_\tau^n\|_{V^*}^2+|\langle F_\tau^1,\nu_{\tau}^1\rangle|.
\end{eqnarray}
Now by Lemma \ref{lem:usd} and $M_B>L_j$,  we can derive that
\begin{eqnarray*}
\left(\frac{M_A}{2}-\eta\right)\|\nu_{\tau}^{n}\|^2\leq c+\tau\epsilon\sum_{k=1}^{n-1}\|\nu_{\tau}^{k}\|^2,
\end{eqnarray*}
where $c$ is a positive number being independent of $\tau$. Taking $e_k=\|\nu_{\tau}^{n}\|^2$ and $g_k=1$ in Lemma \ref{lem:Grownwell}, we obtain
\begin{eqnarray*}
\max_{0\le n\le N}\|\nu_{\tau}^{n}\|^2\leq C.
\end{eqnarray*}
It follows from \eqref{lemequ:Rothe27} that
\begin{eqnarray*}
\frac{1}{2}\|z_{\tau}^n\|_H^2\leq c+\tau \epsilon \sum_{k=1}^{n-1}C\leq c+TC\epsilon,\quad
\frac{1}{2}\sum_{k=2}^{n}\|z_{\tau}^k-z_{\tau}^{k-1}\|_H^2\leq c+\tau \epsilon \sum_{k=1}^{n-1}C\leq c+TC\epsilon
\end{eqnarray*}
and
\begin{eqnarray*}
\left(\frac{M_A}{2}+M_B-L_j\right)\sum_{k=2}^{n}\|\nu_{\tau}^{k-1}-\nu_{\tau}^{k}\|^2 \leq  c+\tau \epsilon \sum_{k=1}^{n-1}C
\leq c+TC\epsilon.
\end{eqnarray*}
Therefore, there exists a constant $C>0$ being independent of $\tau$ such that
\begin{eqnarray*}
\max_{2\le k\le N}\|z_{\tau}^k\|_H\leq C,\quad \max_{2\le k\le N}\|\nu_{\tau}^k\|\leq C,\quad
\sum_{k=2}^{N}\|\nu_{\tau}^k-\nu_{\tau}^{k-1}\|\leq C, \quad \sum_{k=2}^{N}\|z_{\tau}^k-z_{\tau}^{k-1}\|_H\leq C.
\end{eqnarray*}
This completes the proof.
\end{proof}
\begin{remark}
From Lemmas \ref{lem:usd} and \ref{lem:Rothe2}, we can see that there exists a constant $C>0$ being independent of $\tau$ such that
\begin{eqnarray*}
\max_{1\le k\le N}\|z_{\tau}^k\|_H\leq C,\quad \max_{1\le k\le N}\|\nu_{\tau}^k\|\leq C,\quad
\sum_{k=1}^{N}\|\nu_{\tau}^k-\nu_{\tau}^{k-1}\|\leq C, \quad \sum_{k=2}^{N}\|z_{\tau}^k-z_{\tau}^{k-1}\|_H\leq C.
\end{eqnarray*}
\end{remark}
Following the work of Mig\'{o}rski and Zeng \cite{zengshengda 2017}, we define piecewise affine functions $u_{\tau}$, $\nu_{\tau}$, $z_\tau$ and piecewise constant functions $\tilde{u}_\tau$, $\tilde{\nu}_\tau$, $\tilde{z}_\tau$, $f_\tau$ on the time interval $I$ as follows:
\begin{align*}\label{fun}
\begin{cases}
u_\tau(t)=u_\tau^k+\frac{t-t_k}{\tau}(u_\tau^k-u_\tau^{k-1}),\quad \forall t \in(t_{k-1},t_k],\\
\nu_\tau(t)=\nu_\tau^k+\frac{t-t_k}{\tau}(\nu_\tau^k-\nu_\tau^{k-1}),\quad \forall t \in(t_{k-1},t_k],\\
z_\tau(t)=z_\tau^k+\frac{t-t_k}{\tau}(z_\tau^k-z_\tau^{k-1}),\quad \forall t \in(t_{k-1},t_k],
\end{cases}
\end{align*}

\begin{align*}
\tilde{u}_\tau(t)=
\begin{cases}
u_\tau^k,\quad \forall t \in(t_{k-1},t_k],\\
u_0,\quad t =0,
\end{cases}
\quad
\tilde{\nu}_\tau(t)=
\begin{cases}
\nu_\tau^k,\quad \forall \;t \in(t_{k-1},t_k], \\
\nu_0,\quad t =0,
\end{cases}
\end{align*}

\begin{align*}
\tilde{z}_\tau(t)=
\begin{cases}
z_\tau^k,\quad \forall t \in(t_{k-1},t_k],\\
z_\tau^0,\quad t =0,
\end{cases}
\quad
\tilde{f}_\tau(t)=
\begin{cases}
f_\tau^k,\quad \forall t \in(t_{k-1},t_k],\\
f_\tau^0,\quad t =0.
\end{cases}
\end{align*}

\begin{lemma}
If \eqref{pro:pro1}-\eqref{pro:pro8} are satisfied, then there exist two constants $\tau_0>0$ and $C>0$ being independent of $\tau$ such that, for any $\tau\in (0,\tau_0)$,
\begin{align*}
\begin{cases}
\max_{t\in I}\|u_\tau(t)\|\leq C,\;
\max_{t\in I}\|\nu_\tau(t)\|\leq C,\;
\max_{t\in I}\|z_\tau(t)\|\leq C,\\
\max_{t\in I}\|u_\tau(t)-\tilde{u}_\tau(t)\|\leq \frac{C}{N},\;
\max_{t\in I}\|\nu_\tau(t)-\tilde{\nu}_\tau(t)\|\leq \frac{C}{N},
\end{cases}
\end{align*}
where $\tau=\frac{T}{N}$. Furthermore, for any $t_1,t_2\in I$,
\begin{align*}
\|u_\tau(t_1)-\tilde{u}_\tau(t_2)\|\leq C|t_1-t_2|,\quad
\|\nu_\tau(t_1)-\tilde{\nu}_\tau(t_2)\|\leq C|t_1-t_2|.
\end{align*}
\end{lemma}
\begin{proof}
 According to definitions of piecewise affine functions $u_{\tau}$, $\nu_{\tau}$, $z_\tau$ and piecewise constant functions $\tilde{u}_\tau$, $\tilde{\nu}_\tau$, $\tilde{z}_\tau$,  the above conclusions can be easily obtained by using Lemma \ref{lem:Rothe2}.
\end{proof}
\begin{lemma}\label{lem:ae} (\cite{Hytonen 2016})
If $V$ is a Banach space and $\{v_n(t)\}_{n=0}^\infty \subseteq \mathbb{V}$ and $v_n(t)\rightarrow v(t)$ in $\mathbb{V}$, then there exists a subsequence $\{v_{n_i}(t)\}_{i=0}^\infty$ such that $v_{n_i}(t) \rightarrow v(t)$ in $V$ for a.e. $t\in I$.
\end{lemma}

\begin{theorem}\label{theorem1}
If conditions \eqref{pro:pro1}-\eqref{pro:pro9} are satisfied and $L_j<M_B<2C_j$, then there exists a function $u(t)\in C(I;V)$ with $\dot{u}(t)\in C(I;V)\bigcap L^{\infty}I;V)$ and $\ddot{u}(t)\in C(I;V)\bigcap L^{\infty}(I;V)$ such that
\begin{eqnarray*}\label{the}
u_\tau(t)\rightarrow\; u(t) \; in \; C(I;V)\; as \; \tau\rightarrow 0, \quad \nu_\tau(t)\rightarrow \; \dot{u}(t)\; in \; C(I;V) \; as \; \tau\rightarrow 0.
\end{eqnarray*}
Moreover,
\begin{align}
\begin{cases}
\tilde{u}_\tau(t)\rightarrow u(t)\quad in\; V\; for \; a.e.\; t\; \in I,\\
\tilde{\nu}_{\tau_k}(t)\rightarrow \tilde{\nu}(t) \quad in\; V\; for \; a.e.\; t\; \in I,\\
\dot{\nu}_{\tau_j}(t) \rightarrow \ddot{u}(t)\; \quad in\; V^*\; for \; a.e.\; t\; \in I,\\
\tilde{f}_\tau(t)\rightarrow f(t)\; \quad in\; V^*\; for \; a.e.\; t\; \in I,
\end{cases}
\end{align}
where $u$ is the unique solution of Problem \ref{problem:VI}.
\end{theorem}
\begin{proof}
As $V$ is a Banach space, we know that $W^{k,p}(I;V)$ and $C(I;V)$ are both Banach spaces. Thus we only need to prove that $\{u_{\tau_n}(t)\}_{n=0}^{\infty}$, $\{\nu_{\tau_n}(t)\}_{n=0}^{\infty}$, $\{\tilde{u}_{\tau_n}(t)\}_{n=0}^{\infty}$ and $\{\tilde{\nu}_{\tau_n}(t)\}_{n=0}^{\infty}$ are all Cauchy sequences for any  $\{\tau_n\}_{n=0}^{\infty}$ in $(0,+\infty)$ with $\tau_n\to 0$. Since $\dot{\nu}_{\tau}(t)=\tilde{z}_\tau(t)$, we can rewrite \eqref{equ:variRothe1} as follows:
\begin{eqnarray}\label{the1}
&&(\dot{\nu}_{\tau}(t),v-\tilde{\nu}_{\tau}(t))_H+\langle A\tilde{u}_\tau(t),v-\tilde{\nu}_{\tau}(t)\rangle
+\langle B\tilde{\nu}_\tau(t),v-\tilde{\nu}_{\tau}(t)\rangle+j(\tilde{\nu}_\tau(t),v)
-j(\tilde{\nu}_\tau(t),\tilde{\nu}_\tau(t))\nonumber\\
&\geq& \langle\tilde{f}_\tau(t),v-\tilde{\nu}_\tau(t)\rangle,\;\forall v \in V,\; \forall t \in I
\end{eqnarray}
with $\tau=\frac{T}{N}$.  Taking $\rho=\frac{T}{M}$ with $M>N$ in \eqref{the1} yields
\begin{eqnarray}\label{the2}
&&(\dot{\nu}_{\rho}(t),v-\tilde{\nu}_{\rho}(t))_H+\langle A\tilde{u}_\rho(t),v-\tilde{\nu}_{\rho}(t)\rangle
+\langle B\tilde{\nu}_\rho(t),v-\tilde{\nu}_{\rho}(t)\rangle+j(\tilde{\nu}_\rho(t),v)
-j(\tilde{\nu}_\rho(t),\tilde{\nu}_\rho(t))\nonumber\\
&\geq &\langle \tilde{f}_\rho(t),v-\tilde{\nu}_\rho(t)\rangle,\;\forall v \in V,\; \forall t \in I.
\end{eqnarray}
Taking  $v=\tilde{\nu}_\rho(t)$ in inequality \eqref{the1} and $v=\tilde{\nu}_\tau(t)$ in inequality \eqref{the2}, respectively, and adding \eqref{the1} to \eqref{the2}, we obtain
 \begin{eqnarray*}\label{the3}
&&(\dot{\nu}_{\tau}(t)-\dot{\nu}_{\rho}(t),\tilde{\nu}_{\rho}(t)-\tilde{\nu}_{\tau}(t))_H
+\langle A\tilde{u}_\tau(t)-A\tilde{u}_{\rho}(t)+B\tilde{\nu}_\tau(t)-B\tilde{\nu}_{\rho}(t),\tilde{\nu}_{\rho}(t)-\tilde{\nu}_{\tau}(t)\rangle\\
&&\qquad \mbox{}+L_j\|\tilde{\nu}_\rho(t)-\tilde{\nu}_\tau(t)\|^2 \geq \langle \tilde{f}_\tau(t)-\tilde{f}_\rho(t),\tilde{\nu}_{\rho}(t)-\tilde{\nu}_\tau(t)\rangle,\quad  \forall t \in I.
\end{eqnarray*}
It follows that
 \begin{eqnarray*}\label{the5}
&&(\dot{\nu}_{\tau}(t)-\dot{\nu}_{\rho}(t),\nu_{\tau}(t)-\nu_{\rho}(t))_H
+\langle Au_\tau(t)-Au_{\rho}(t)+B\tilde{\nu}_\tau(t)-B\tilde{\nu}_{\rho}(t),\tilde{\nu}_{\tau}(t)-\tilde{\nu}_{\rho}(t)\rangle\nonumber\\
&\leq& \langle \dot{\nu}_{\tau}(t)-\dot{\nu}_{\rho}(t),\tilde{\nu}_{\rho}(t)-\nu_{\rho}(t)-\tilde{\nu}_{\tau}(t)+\nu_{\tau}(t)\rangle
+\langle A\tilde{u}_\tau(t)-Au_\tau(t)-A\tilde{u}_\rho(t)+Au_{\rho}(t),\tilde{\nu}_{\rho}(t)-\tilde{\nu}_{\tau}(t)\rangle\nonumber\\
&&\mbox{}-\langle \tilde{f}_\tau(t)-\tilde{f}_\rho(t),\tilde{\nu}_{\rho}(t)-\tilde{\nu}_\tau(t)\rangle
+L_j\|\tilde{\nu}_\rho(t)-\tilde{\nu}_\tau(t)\|^2,\quad \forall  t \in I.
\end{eqnarray*}
Employing the hypothesis of $A$, one has
\begin{eqnarray}\label{the6}
&&\frac{1}{2}\frac{d}{dt}\|\nu_{\rho}(t)-\nu_{\tau}(t)\|_H^2+\frac{1}{2}\frac{d}{dt}(Au_\tau(t)-Au_{\rho}(t),u_{\tau}(t)-u_{\rho}(t))+
(B\tilde{\nu}_\tau(t)-B\tilde{\nu}_{\rho}(t),\tilde{\nu}_{\tau}(t)-\tilde{\nu}_{\rho}(t)) \nonumber\\
&\leq&(\|\dot{\nu}_\tau(t)\|+\|\dot{\nu}_{\rho}(t)\|)(\|\tilde{\nu}_\tau(t)-\nu_\tau(t)\|+\|\tilde{\nu}_{\rho}(t)-\nu_{\rho}(t)\|)\nonumber\\
&&\mbox{}+L_A(\|u_{\rho}(t)-\tilde{u}_\rho(t)\|+\|u_\tau(t)-\tilde{u}_\tau(t)\|)(\|\tilde{\nu}_{\tau}(t)\|+\|\tilde{\nu}_{\rho}(t)\|)\nonumber\\
&&\mbox{}+\|\tilde{f}_\tau(t)-\tilde{f}_\rho(t)\|\|\tilde{\nu}_{\rho}(t)-\tilde{\nu}_\tau(t)\|
 +L_j\|\tilde{\nu}_\rho(t)-\tilde{\nu}_\tau(t)\|^2, \quad \forall t \in I.
\end{eqnarray}
By Lemmas \ref{lem:usd} and \ref{lem:Rothe2}, it follows from Lemma 10 of \cite{zengshengda 2017} that there exists a constant $C>0$ being independent of $\tau$ such that
 \begin{align}\label{the7}
 \begin{cases}
\|\dot{\nu}_\tau(t)\|=\|\tilde{z}_\tau(t)\| \leq C,\;
\|\dot{\nu}_\rho(t)\|=\|\tilde{z}_\rho(t)\| \leq C,\;
\|\dot{u}_\tau(t)\|=\|\tilde{\nu}_\tau(t)\| \leq C,\\
\|\dot{u}_\rho(t)\|=\|\tilde{\nu}_\rho(t)\| \leq C,\;
\|\tilde{\nu}_{\rho}(t)-\tilde{\nu}_\tau(t)\|\leq C,\\
\|\tilde{\nu}_{\rho}(t)-\nu_{\rho}(t)\| = \left\|\frac{t-t_k}{\rho}(\nu_\rho^k-\nu_\rho^{k-1})\right\|\leq |t-t_k|\|z_\rho^k\|
\leq \frac{C}{M}, \\
\|\tilde{\nu}_{\tau}(t)-\nu_{\tau}(t)\| =\left \|\frac{t-t_k}{\tau}(\nu_\tau^k-\nu_\tau^{k-1})\right\|\leq |t-t_k|\|z_\tau^k\|
\leq \frac{C}{N}, \\
\|\tilde{u}_{\rho}(t)-u_{\rho}(t)\| = \left\|\frac{t-t_k}{\rho}(u_\rho^k-u_\rho^{k-1})\right\|\leq |t-t_k|\|v_\rho^k\|
\leq \frac{C}{M}, \\
\|\tilde{u}_{\tau}(t)-u_{\tau}(t)\| = \left\|\frac{t-t_k}{\tau}(u_\tau^k-u_\tau^{k-1})\right\|\leq |t-t_k|\|v_\tau^k\|
\leq \frac{C}{N}, \\
\|f(t)-\tilde{f}_{\rho}(t)\|\leq \sqrt{\rho}\left\|\frac{df}{ds}\right\|_\mathbb{V^*}, \;
\|f(t)-\tilde{f}_{\tau}(t)\|\leq \sqrt{\tau}\left\|\frac{df}{ds}\right\|_\mathbb{V^*}
\end{cases}
\end{align}
for all $t\in (t_{k-1},t_{k}]$.
Thus, it follows from inequality \eqref{the6} that
 \begin{eqnarray}\label{the8}
&&\frac{1}{2}\frac{d}{dt}\|\nu_{\rho}(t)-\nu_{\tau}(t)\|_H^2
+\frac{1}{2}\frac{d}{dt}\langle Au_\tau(t)-Au_{\rho}(t),u_{\tau}(t)-u_{\rho}(t)\rangle
+(M_B-L_j)\|\tilde{\nu}_\tau(t)-\tilde{\nu}_{\rho}(t)\|^2\nonumber\\
&\leq&(2C)(\|\tilde{\nu}_\tau(t)-\nu_\tau(t)\|+\|\tilde{\nu}_{\rho}(t)-\nu_{\rho}(t)\|)
+(2CL_A)(\|u_{\rho}(t)-\tilde{u}_\rho(t)\|+\|u_\tau(t)-\tilde{u}_\tau(t)\|)
+C\|\tilde{f}_\tau(t)-\tilde{f}_\rho(t)\|\nonumber\\
&\leq&(2C^2+2C^2L_A)\left(\frac{1}{M}+\frac{1}{N}\right)+C\|\tilde{f}_\tau(t)-f(t)\|+C\|f(t)-\tilde{f}_\rho(t)\|\nonumber\\
&\leq&(2C^2+2C^2L_A)\left(\frac{1}{M}+\frac{1}{N}\right)+C\left(\sqrt{\frac{T}{M}}+\sqrt{\frac{T}{N}}\right)\left\|\frac{df}{ds}\right\|_\mathbb{V^*}, \quad \forall t\in I.
\end{eqnarray}
Integrating both sides on $(0,t]$ of \eqref{the8}, we have
\begin{eqnarray}\label{the9}
&&\|\nu_{\rho}(t)-\nu_{\tau}(t)\|_H^2
+M_A\|u_\tau(t)-u_{\rho}(t)\|^2-\|\nu_{\rho}(0)-\nu_{\tau}(0)\|^2-M_A\|u_{\rho}(0)-u_{\tau}(0)\|^2\nonumber\\
&\leq&t(4C^2+4C^2L_A)\left(\frac{1}{M}+\frac{1}{N}\right)
+2tC\left(\sqrt{\frac{T}{M}}+\sqrt{\frac{T}{N}}\right)\left\|\frac{df}{ds}\right\|_\mathbb{V^*},\quad \forall t\in I.
\end{eqnarray}
From \eqref{the8}, we know that
\begin{eqnarray}\label{my1}
(M_B-L_j)\|\tilde{\nu}_\tau(t)-\tilde{\nu}_{\rho}(t)\|^2
\leq(2C^2+2C^2L_A)\left(\frac{1}{M}+\frac{1}{N}\right)+C\left(\sqrt{\frac{T}{M}}+\sqrt{\frac{T}{N}}\right)\left\|\frac{df}{ds}\right\|_\mathbb{V^*}, \;\forall t\in I.
\end{eqnarray}
Integrating both sides on $(0,T]$ of \eqref{my1}, one has
\begin{eqnarray}\label{the10}
(M_B-L_j)\|\tilde{\nu}_{\tau}(t)-\tilde{\nu}_{\rho}(t)\|_{\mathbb{V}}^2\leq
T(4C^3+4C^3L_A)\left(\frac{1}{M}+\frac{1}{N}\right)
+2TC\left(\sqrt{\frac{T}{M}}+\sqrt{\frac{T}{N}}\right)\left\|\frac{df}{ds}\right\|_\mathbb{V^*}.
\end{eqnarray}
Since $C$ is independent of $\tau$ and $\rho$, we deduce that $\{\nu_{\tau}(t)\}$ and $\{u_{\tau}(t)\}$ are Cauchy sequences in $C(I;H)$ and $C(I;V)$, respectively. Thus, the completeness of $C(I;V)$ and $C(I;H)$ implies that there exist $u(t)\in C(I;V)$ and $\nu(t) \in C(I;H)$ such that
\begin{eqnarray*}\label{the11}
u_{\tau}(t)\rightarrow u(t)\; \mbox{in} \; C(I;V), \quad \nu_{\tau}(t)\rightarrow \nu(t) \;\mbox{in} \; C(I;H) \quad as \; \tau \rightarrow 0.
\end{eqnarray*}
It follows from \eqref{the7} that $\tilde{u}_{\tau}(t)\rightarrow u(t)$ in $C(I;V)$ and $\tilde{\nu}_{\tau}(t)\rightarrow \nu(t)$ in $C(I;H)$. By the inequality \eqref{the10}, we obtain $\tilde{\nu}_{\tau}(t)\rightarrow \nu(t)$ in $\mathbb{V}$.
From Theorem 2.39 of \cite{Migorski 2012} and Lemma \ref{lem:Rothe2}, we have $\nu(t)=\dot{u}(t)$ for a.e. $t\in I$. Now Lemma 1.3.15 of \cite{Kacur 1986} shows that $\dot{u}(t)\in L^{\infty}(I;V)$,  $\dot{\nu}(t) \in L^{\infty}(I;H)$,  $\tilde{\nu}_{\tau}(t)\rightarrow\dot{u}(t)$ weakly in $V$ for a.e. $t\in I$ and $\dot{\nu}_\tau(t)\rightarrow \ddot{u}(t)$ weakly in $\mathbb{H}$.

Since $H$ is compactly embedded in $V^*$ and $\dot{\nu}_\tau(t)\rightarrow \ddot{u}(t)$ weakly in $\mathbb{H}$, we have $\dot{\nu}_\tau(t)\rightarrow \ddot{u}(t)$ strongly in $\mathbb{V^*}$. Moreover, for sequences $\{\tilde{\nu}_\tau(t)\}$ and $\{\dot{\nu}_\tau(t)\}$, we have
$\tilde{\nu}_{\tau}(t)\rightarrow \nu(t)$ in $ \mathbb{V}$ and $\dot{\nu}_\tau(t)\rightarrow \ddot{u}(t)$ in $\mathbb{V^*}$. Thus, by Lemma \ref{lem:ae}, there exist subsequences $\{\tilde{\nu}_{\tau_k}(t)\}\subset \{\tilde{\nu}_{\tau}(t)\}$ and $\{\dot{\nu}_{\tau_j}(t)\}\subset\{\dot{\nu}_{\tau}(t)\}$ such that $\{\tilde{\nu}_{\tau_k}(t)\}$ converges to $\nu(t)$ in $H$ for a.e. $t\in I$ and $\{\dot{\nu}_{\tau_j}(t)\}$ converges to $\ddot{u}(t)$ in $V^*$ for a.e. $t\in I$.

In a conclusion, we have
\begin{align}\label{will}
\begin{cases}
\tilde{u}_\tau(t)\rightarrow u(t) \quad
\tilde{\nu}_{\tau_k}(t)\rightarrow \tilde{\nu}(t) \; & \mbox{in}\; V\; \mbox{for} \; a.e.\; t\; \in I,\\
\dot{\nu}_{\tau_j}(t) \rightarrow \ddot{u}(t)\quad
\tilde{f}_\tau(t)\rightarrow f(t)\; & \mbox{in}\; V^*\; \mbox{for} \; a.e.\; t\; \in I.
\end{cases}
\end{align}
From the continuity of operators $A$, $B$ and the functional $j$, it follows from \eqref{will} that there exists a sequence $\{\tau_n\}$ such that
\begin{align}\label{the14}
\begin{cases}
\langle \dot{\nu}_{\tau_n}(t),v-\tilde{\nu}_{\tau_n}(t)\rangle\rightarrow \langle\ddot{u}(t),v-\dot{u}(t)\rangle,\\
\langle A\tilde{u}_{\tau_n}(t),v-\tilde{\nu}_{\tau_n}(t)\rangle\rightarrow \langle A(u(t)),v-\dot{u}(t)\rangle,\\
\langle B\tilde{\nu}_{\tau_n}(t),v-\tilde{\nu}_{\tau_n}\rangle\rightarrow \langle B(\dot{u}(t)),v-\dot{u}(t)\rangle,\\
\langle \tilde{f}_{\tau_n}(t),v-\tilde{\nu}_{\tau_n}\rangle\rightarrow \langle f(t),v-\dot{u}(t)\rangle,\\
j(\tilde{\nu}_{\tau_n}(t),v)-j(\tilde{\nu}_{\tau_n}(t),\tilde{\nu}_{\tau_n}(t))\rightarrow j(\dot{u}(t),v(t))-j(\dot{u}(t),\dot{u}(t))
\end{cases}
\end{align}
for a.e. $t\in I$.  Letting $\tau=\tau_n$ in \eqref{the1} and  $n\rightarrow \infty$, we have
\begin{align}\label{the15}
\begin{cases}
\langle \ddot{u}(t)+Au(t)+B\dot{u}(t)-f(t),v-\dot{u}(t)\rangle+j(\dot{u}(t),v)
-j(\dot{u}(t),\dot{u}(t))\geq 0,\;\forall v\in V,\;\mbox{a.e.}\;t \in I,\\
\;u(0)=u_0,\;\dot{u}(0)=v_0.
\end{cases}
\end{align}
Now we prove the uniqueness of solutions to Problem \ref{problem:VI}. Assuming that $u_1$and $u_2$ are two solutions of Problem \ref{problem:VI}, one has
\begin{eqnarray}\label{1}
\langle\ddot{u}_1(t)+Au_1(t)+B\dot{u_1}(t)-f(t),v-\dot{u}_1(t)\rangle_{V^*\times V}+j(\dot{u}_1(t),v)-j(\dot{u}_1(t),\dot{u}_1(t))\geq 0
\end{eqnarray}
and
\begin{eqnarray}\label{2}
\langle\ddot{u}_2(t)+Au_2(t)+B\dot{u_2}(t)-f(t),v-\dot{u}_2(t)\rangle_{V^*\times V}+j(\dot{u}_2(t),v)-j(\dot{u}_2(t),\dot{u}_2(t))\geq 0.
\end{eqnarray}
Taking  $v=\dot{u}_2(t)$ and $v=\dot{u}_1(t)$ in \eqref{1} and \eqref{2}, respectively, and then adding \eqref{1} to \eqref{2}, we obtain
\begin{eqnarray*}\label{the16}
\frac{1}{2}\frac{d}{dt}\|\dot{u}_1(t)-\dot{u}_2(t)\|^2+\frac{1}{2}\frac{d}{dt}\langle Au_1(t)-Au_2(t),u_1(t)-u_2(t)\rangle+
\langle B\dot{u}_1(t)-B\dot{u}_2(t),\dot{u}_1(t)-\dot{u}_2(t)\rangle\leq 0
\end{eqnarray*}
and so
\begin{eqnarray*}\label{the17}
\frac{1}{2}\frac{d}{dt}\|\dot{u}_1(t)-\dot{u}_2(t)\|^2+\frac{1}{2}\frac{d}{dt}\|u_1(t)-u_2(t)\|^2+
L_B\|\dot{u}_1(t)-\dot{u}_2(t)\|\leq 0
\end{eqnarray*}
for a.e. $t \in I$. From the initial condition that $u_1(0)=u_2(0)$ and $\dot{u}_1(0)=\dot{u}_2(0)$.
\end{proof}

\section{The solution of dynamic contact problem}
\setcounter{equation}{0}

In this section, we apply the result presented in Section 3 to study Problem \ref{problem:Pv}.  To this end, we assume that the viscosity operator satisfies the following conditions:
\begin{align}\label{assume:A}
\begin{cases}
\mathbb{A}:\Omega\times \mathbb{S}^d\rightarrow \mathbb{S}^d.\\
\mbox{There exists} \; \mathbb{L}_\mathbb{A}\;\mbox{such that} \;
|\mathbb{A}(x,\varepsilon_1)-\mathbb{A}(x,\varepsilon_2)|\leq\mathbb{L}_{\mathbb{A}}|\varepsilon_1-\varepsilon_2| \; \mbox{for all} \; \varepsilon_1,\varepsilon_2\in \mathbb{S}^d.\\
\mbox{There exists} \; \mathbb{M}_\mathbb{A} \; \mbox{such that} \;
(\mathbb{A}(x,\varepsilon_1)-\mathbb{A}(x,\varepsilon_1)):(\varepsilon_1-\varepsilon_2)\geq \mathbb{M}_\mathbb{A}|\varepsilon_1-\varepsilon_2|^2 \; \mbox{for all} \; \varepsilon_1,\;\varepsilon_2\in \mathbb{S}^d.\\
\mathbb{A}(\cdot,\varepsilon)\; \mbox{is Lebesgue measurable for all} \; \varepsilon\in \mathbb{S}^d \;\mbox{and}\;\mathbb{A}(x,0)\in Q.
\end{cases}
\end{align}
Moreover, we assume that the elasticity operator satisfies the following conditions:
\begin{align}\label{assume:G}
\begin{cases}
\mathbb{G}(\cdot)\in\mathcal{L}(\mathbb{S}^d,\mathbb{S}^d), \mbox{i.e., there exists\;} \mathbb{L}_\mathbb{G} \mbox{\;such that\;}
|\mathbb{G}(\varepsilon))|\leq\mathbb{L}_{\mathbb{G}}|\varepsilon|,\;\mbox{for all}\; \varepsilon\in \mathbb{S}^d .\\
\mbox{There exists\;} \mathbb{M}_\mathbb{G} \mbox{\;such that\;}
(\mathbb{G}(\varepsilon_1)-\mathbb{G}(\varepsilon_1)):(\varepsilon_1-\varepsilon_2)\geq
\mathbb{M}_\mathbb{G}|\varepsilon_1-\varepsilon_2|^2 \; \mbox{for all}\; \varepsilon_1,\;\varepsilon_2\in \mathbb{S}^d.\\
\mathbb{G}\varepsilon_1:\varepsilon_2=\mathbb{G}\varepsilon_2:\varepsilon_1  \; \mbox{for all}\; \varepsilon_1,\;\varepsilon_2\in \mathbb{S}^d\;\mbox{and }\mathbb{G}(0)\in Q.
\end{cases}
\end{align}
The body force $f_0$, surface traction $g$, coefficient of friction $\mu$, adhesion field $\beta$, initial conditions $u_0,\;v_0$ and mass density $\rho$ have the following properties:
\begin{align}\label{assume:others}
\begin{cases}
f_0\in H^2(I;L^2(\Omega,\mathbb{R}^d)).\\
g\in H^2(I;L^2(\Gamma_2,\mathbb{R}^d)).\\
\mu\in L^\infty(\Gamma_3,\mathbb{R}),\quad \mu(x)\geq0\; for\; a.e.\; x\in \Gamma_3.\\
\beta\in L^\infty(\Gamma_3,\mathbb{R}),\quad \beta(x)\geq \beta^*>0\; for \; a.e. \; x\in \Gamma_3.\\
\rho\in L^\infty(\Omega,\mathbb{R}),\quad \beta(x)\geq \rho^*>0\; for \; a.e.\; x\in \Omega.\\
u_0\in V\;v_0\in H.
\end{cases}
\end{align}
From \eqref{assume:others}, we can define $f(t)\in V^*$ by setting
$$\langle f(t),v \rangle_{V^*\times V}=(f_0(t),v)_H+\int_{\Gamma_2 }g(t)vd\Gamma.$$
\begin{theorem}
Let assumptions \eqref{assume:A}, \eqref{assume:G} and \eqref{assume:others} hold. Suppose that there exists a constant $\alpha_0>0$ depending on $\Gamma_3$, such that
$$\|\beta\|_{L^\infty(\Gamma_3)}(\|\mu\|_{L^\infty(\Gamma_3)}+1)<\alpha_0$$
and
$$c_\gamma^2\|\beta\|_{L^\infty(\Gamma_3,\mathbb{R}^d)}(\|\mu\|_{L^\infty(\Gamma_3,\mathbb{R}^d)}+1)<\mathbb{M}_\mathbb{A}
<2c_\gamma^2\|\beta\|_{L^\infty(\Gamma_3)}(\|\mu\|_{L^\infty(\Gamma_3)}+1),$$
where $ c_\gamma $ is a constant.  Then Problem \ref{problem:Pv} has a unique solution $u\in C(I;V)$ satisfying
$$\dot{u}\in C(I;V)\bigcap L^{\infty}(I;V), \quad \ddot{u}\in C(I;V)\bigcap L^{\infty}(I;V).$$
\end{theorem}
\begin{proof}
The proof is based on Theorem \ref{theorem1}. Let $V$ denote the space defined in \eqref{Vdefine} and $H=L^2(\Omega,\mathbb{R}^d)$. Then there are an evolution triple of spaces $V\hookrightarrow H\hookrightarrow V^*$ and a compactly embedding operator $i:V\hookrightarrow H$. Define two operators $A,\;B:V\rightarrow V^*$, an inner product $((\cdot,\cdot))_H$ and a functional $f:V\rightarrow V^*$ by setting
\begin{align*}
\begin{cases}
((u,v))_H=(\rho(x) u,v)_H.\\
\langle Au,v\rangle=\langle \mathbb{G}(\varepsilon(u)),\varepsilon(v)\rangle_Q.\\
\langle Bu,v\rangle=\langle\mathbb{A}(\varepsilon(u)),\varepsilon(v)\rangle_Q.\\
\langle f,v\rangle=L(v).
\end{cases}
\end{align*}
Obviously, $((\cdot,\cdot))_H$ and $(\cdot,\cdot)_H$ are equivalent inner products due to the assumption of mass density $\rho$.
Thus, we know that Problem \ref{problem:Pv} can be transformed as follows:
\begin{align*}
\begin{cases}
\langle \ddot{u}(t)+Au(t)+B\dot{u}(t)-f(t),v-\dot{u}(t)\rangle+j(\dot{u}(t),v)-j(\dot{u}(t),\dot{u}(t))\geq 0,\quad \forall v\in V.\\
u(0)=u_0,\;\dot{u}(0)=v_0.
\end{cases}
\end{align*}
Thus, we only need to verify that all the conditions \eqref{pro:pro1}-\eqref{pro:pro9} are satisfied. Clearly, \eqref{pro:pro1}-\eqref{pro:pro5} and \eqref{pro:pro8} are met. Now we turn to check the remaining conditions. Since
$$j(u,v)=\int_{\Gamma_3}\beta|u_\nu|(\mu|v_\tau-v^*|+v_\nu)d\Gamma$$
and
$$|\lambda v_1+(1-\lambda)v_2-v^*|\leq\lambda| v_1-v^*|+(1-\lambda)|v_2-v^*|,$$
we deduce that $j(u,\cdot)$ is a proper convex functional and for any $v_1,v_2\in V$, 
\begin{eqnarray}\label{j_ass1}
j(g,v_1)-j(g,v_2)
&=&\int_{\Gamma_3}\beta|g_\nu|(\mu|v_{1,\tau}-v^*|+v_{1,\nu})d\Gamma-\int_{\Gamma_3}\beta|g_\nu|(\mu|v_{2,\tau}-v^*|+v_{2,\nu})d\Gamma\nonumber\\
&=&\int_{\Gamma_3}\beta|g_\nu|(\mu|v_{1,\tau}-v^*|-\mu|v_{2,\tau}-v^*|+v_{1,\nu}-v_{2,\nu})d\Gamma\nonumber\\
&\leq&\int_{\Gamma_3}\beta|g_\nu|(\mu|v_{1,\tau}-v_{2,\tau}|+|v_{1,\nu}-v_{2,\nu})|d\Gamma\nonumber\\
&\leq&\|\beta\|_{L^\infty(\Gamma_3,\mathbb{R})}\|\mu\|_{L^\infty(\Gamma_3)}\|g_{\nu}\|_{L^2(\Gamma_3)}\|v_{1,\tau}-v_{2,\tau}\|_{L^2(\Gamma_3)}\nonumber\\
&&\mbox{}+\|\beta\|_{L^\infty(\Gamma_3,\mathbb{R})}\|g_{\nu}\|_{L^2(\Gamma_3)}\|v_{1,\nu}-v_{2,\nu}\|_{L^2(\Gamma_3)}\nonumber\\
&\leq&\|\beta\|_{L^\infty(\Gamma_3,\mathbb{R})}(\|\mu\|_{L^\infty(\Gamma_3)}+1)\|g_{\nu}\|_{L^2(\Gamma_3)}\|v_{1,\nu}-v_{2,\nu}\|_{L^2(\Gamma_3)}\nonumber\\
&\leq&\|\beta\|_{L^\infty(\Gamma_3,\mathbb{R})}(\|\mu\|_{L^\infty(\Gamma_3)}+1)\|g\|_{L^2(\Gamma_3;\mathbb{R}^d)}\|v_1-v_2\|_{L^2(\Gamma_3;\mathbb{R}^d)}.
\end{eqnarray}
Recall the trace theorem  (\cite{Roubicek 2013}), we know that there exists a constant $c_\gamma>0$ such that
$$\|u\|_{L^2(\Gamma_3,\mathbb{R}^d)}\leq \|u\|_{L^2(\Omega;\mathbb{R}^d)}\leq c_\gamma\|u\|_V, \quad \forall u\in L^2(\Gamma_3,\mathbb{R}^d).$$
Then the inequality \eqref{j_ass1} can be transformed as follows:
\begin{eqnarray*}
j(g,v_1)-j(g,v_2) &=&\int_{\Gamma_3}\beta|g_\nu|(\mu|v_{1,\tau}-v^*|+v_{1,\nu})d\Gamma-\int_{\Gamma_3}\beta|g_\nu|(\mu|v_{2,\tau}-v^*|+v_{2,\nu})d\Gamma\nonumber\\
&\leq& c_\gamma^2\|\beta\|_{L^\infty(\Gamma_3)}(\|\mu\|_{L^\infty(\Gamma_3)}+1)\|g\|_V\|v_1-v_2\|_V.
\end{eqnarray*}
Thus the condition \eqref{pro:pro7} holds. Similarly, we have
\begin{eqnarray*}\label{j_ass}
j(g_1,v_2)-j(g_1,v_1)+j(g_1,v_1)-j(g_2,v_2)
\leq c_\gamma^2\|\beta\|_{L^\infty(\Gamma_3,\mathbb{R}^d)}(\|\mu\|_{L^\infty(\Gamma_3,\mathbb{R}^d)}+1)\|g_1-g_2\|_V\|v_1-v_2\|_V
\end{eqnarray*}
and so the condition \eqref{pro:pro6} is true. Therefore, we verify that all the conditions of Theorem \ref{theorem1} are satisfied and so Problem \ref{problem:Pv} is uniquely solvable.
\end{proof}

\section{Fully discrete approximation}
\setcounter{equation}{0}

To fully discretize the hyperbolic quasi-variational inequality \eqref{equ:Rothe}, in this section,  we shall use finite-dimensional space $V^h$ to approximate $V$, where $V^h$ can be constructed by the finite element method. Let $u^h$ denote the element in $V^h$, where $h$ is the maximal diameter of the elements.
We also adopt the partition of the time: $0 = t_0 < t_l <\cdots < t_N = T$ and $\tau =t_k- t_{k-1} $ for $k=1,2,\cdots,N$. For a function $w_{\tau}^h(t)\in C(I,V^h)$, we substitute $w_{\tau}^{h,k}$ for $w^h_\tau(t_k)$. For a sequence $\{w_\tau^{h,k}\}_{k=0}^{N}$, we define
$$
\delta w_{\tau}^{h,k} =\frac{w_{\tau}^{h,k} - w_{\tau}^{h,k-1}}{\tau}, \quad \delta^2 w_{\tau}^{h,k}=\frac{\delta w_{\tau}^{h,k}-\delta w_{\tau}^{h,k-1}}{\tau}.
$$

In this section, we do not adopt the repeated index to represent the summation. Following the forward Euler method, a fully discrete scheme can be constructed as follows.
\begin{problem}\label{problem:FullyDiscretize}
 Find $\{\nu_{\tau}^{h,k}\}_{k=1}^{N}\subset V^h$ such that $u_{\tau}^{h,0}\in V^h$, $\delta u_{\tau}^{h,0}\in V^h$ and
\begin{align}\label{equ:fullydiscrete}
\begin{cases}
(\delta^2 u_{\tau}^{h,k},v^h-\delta u_{\tau}^{h,k})_H
+\langle A(u_\tau^{h,0}+\tau\sum_{i=1}^{k}\delta u_{\tau}^{h,k})+B\delta u_{\tau}^{h,k}-f_{\tau}^{k},v-\delta u_{\tau}^{h,k}\rangle\\
\qquad \qquad \mbox{} +j(\delta u_{\tau}^{h,k},v^h)-j(\delta u_{\tau}^{h,k},\delta u_{\tau}^{h,k})\geq 0, \quad \forall v^h\in V^h,\; k=1,2,\cdots, N,
\end{cases}
\end{align}
where $u_{\tau}^{h,0}$ and $\delta u_{\tau}^{h,0}$ are finite element approximation of $u_0$ and $\dot{u}_0$, respectively.
\end{problem}

For simplicity, we abbreviate $\delta u_\tau^{h,k}$ and $\delta^2 u_\tau^{h,k}$ to $\nu_\tau^{h,k}$ and $z_\tau^{h,k}$, respectively. Thus, we can rewrite \eqref{equ:fullydiscrete} as follows
\begin{align}\label{equ:brieflyfullydiscrete}
\begin{cases}
(z_\tau^{h,k},v^h-\nu_{\tau}^{h,k})_H
+\langle Au_\tau^{h,k}+B\nu_{\tau}^{h,k}-f_{\tau}^{k},v^h-\nu_{\tau}^{h,k}\rangle
+j(\nu_{\tau}^{h,k},v^h)-j(\nu_{\tau}^{h,k},\nu_{\tau}^{h,k})\geq 0, \; \forall v^h\in V^h,\\
k=1,2,\cdots,N.
\end{cases}
\end{align}
For given $\{\nu_\tau^{h,i}\}_{i=1}^{k-1}$ and $u_\tau^{h,0}$, Lemmas \ref{lem:Rothe} and \ref{lem:Rothe2} show that \eqref{equ:brieflyfullydiscrete} has a unique solution and so Problem \ref{problem:FullyDiscretize} has a unique solution. From now on, we mainly concern with an error estimate for Problem \ref{problem:FullyDiscretize}.

Taking $v=\nu_\tau^{h,k}$ in \eqref{equ:Rothe} at $t_k$ with $k\geq1$, one has
\begin{equation}\label{equ:VIdiscrete2}
\begin{aligned}
(z_\tau^k,\nu_\tau^{h,k}-\nu_{\tau}^k)_H
+\langle Au_\tau^k+B\nu_{\tau}^k-f_{\tau}^k,\nu_\tau^{h,k}-\nu_{\tau}^k \rangle
+j(\nu_{\tau}^k,\nu_\tau^{h,k})-j(\nu_{\tau}^k,\nu_{\tau}^k)\geq 0
\end{aligned}
\end{equation}
Letting $v^h=v^{h,k}\in V^h$ in \eqref{equ:brieflyfullydiscrete} and then adding the above two inequalities, we have
\begin{eqnarray*}
&&(z_\tau^{h,k}-z_\tau^k,\nu_\tau^{h,k}-\nu_\tau^k)_H+\langle B\nu_\tau^{h,k}-B\nu_\tau^k,\nu_\tau^{h,k}-\nu_\tau^k\rangle\nonumber\\
&\leq& \langle Au_\tau^{h,k},v^{h,k}-\nu_\tau^{h,k}\rangle+\langle Au_\tau^{k},\nu_\tau^{h,k}-\nu_\tau^k\rangle\nonumber\\
&&\mbox{}+\langle B\nu_\tau^{h,k},v^{h,k}-\nu_\tau^k\rangle+j(\nu_{\tau}^k,\nu_\tau^{h,k})+j(\nu_{\tau}^{h,k},v^{h,k})\nonumber\\
&&\mbox{}-j(\nu_{\tau}^k,\nu_{\tau}^k)-j(\nu_{\tau}^{h,k},\nu_{\tau}^{h,k})-\langle f_\tau^k,v^{h,k}-\nu_\tau^k\rangle
+(z_\tau^{h,k},v^{h,k}-\nu_\tau^k)\nonumber\\
&=&  \langle Au_\tau^{h,k},v^{h,k}-\nu_\tau^{k}\rangle
+\langle Au_\tau^{h,k}-Au_\tau^{k},\nu_\tau^{k}-\nu_\tau^{h,k}\rangle\nonumber\\
&&\mbox{}+j(\nu_{\tau}^{h,k},v^{h,k})+j(\nu_{\tau}^{k},\nu_\tau^{h,k})
-j(\nu_{\tau}^k,v^{h,k})-j(\nu_{\tau}^{h,k},\nu_{\tau}^{h,k})\nonumber\\
&&\mbox{}+\langle B\nu_\tau^{h,k},v^{h,k}-\nu_\tau^k\rangle+j(\nu_\tau^k,v^{h,k})-j(\nu_\tau^k,\nu_\tau^k)\nonumber\\
&&\mbox{}-\langle f_\tau^k,v^{h,k}-\nu_\tau^k\rangle+(z_\tau^{h,k},v^{h,k}-\nu_\tau^k).
\end{eqnarray*}
Since
$$(z_\tau^{h,k}-z_\tau^k,\nu_\tau^{h,k}-\nu_\tau^k)=\frac{1}{\tau}(\nu_\tau^{h,k}-\nu_\tau^k,\nu_\tau^{h,k}-\nu_\tau^k)-
\frac{1}{\tau}(\nu_\tau^{h,k-1}-\nu_\tau^{k-1},\nu_\tau^{h,k}-\nu_\tau^k)$$
and
\begin{eqnarray*}
&&\langle Au_\tau^{h,k}-Au_\tau^{k},\nu_\tau^{h,k}-\nu_\tau^{k}\rangle\nonumber\\
&=&\frac{1}{2\tau}\langle A u_\tau^{h,k}-Au_\tau^k, u_\tau^{h,k}- u_\tau^k\rangle-
\frac{1}{2\tau}\langle A u_\tau^{h,k}-Au_\tau^{k}, u_\tau^{h,k-1}- u_\tau^{k-1}\rangle
+\frac{1}{2}\langle A \nu_\tau^{h,k}-A\nu_\tau^{k}, \nu_\tau^{h,k}- \nu_\tau^{k}\rangle,
\end{eqnarray*}
we deduce that 
\begin{eqnarray*}
&&\frac{1}{\tau}\|\nu_\tau^{h,k}-\nu_\tau^k\|_H^2+\langle B\nu_\tau^{h,k}-B\nu_\tau^k,\nu_\tau^{h,k}-\nu_\tau^k\rangle\nonumber\\
&\leq&  \langle Au_\tau^{h,k},v^{h,k}-\nu_\tau^{k}\rangle
+\frac{1}{2\tau}\langle Au_\tau^{h,k-1}-Au_\tau^{k-1}, u_\tau^{h,k-1}- u_\tau^{k-1}\rangle
-\frac{1}{2\tau}\langle Au_\tau^{h,k }-Au_\tau^{k }, u_\tau^{h,k}- u_\tau^k\rangle
\nonumber\\
&&\mbox{}-\frac{1}{2}\langle A \nu_\tau^{h,k}-A\nu_\tau^{k}, \nu_\tau^{h,k}- \nu_\tau^{k}\rangle
+j(\nu_{\tau}^{h,k},v^{h,k})+j(\nu_{\tau}^{k},\nu_\tau^{h,k})
-j(\nu_{\tau}^k,v^{h,k})-j(\nu_{\tau}^{h,k},\nu_{\tau}^{h,k})+\langle B\nu_\tau^{h,k},v^{h,k}-\nu_\tau^k\rangle\nonumber\\
&&\mbox{}+j(\nu_\tau^k,v^{h,k})-j(\nu_\tau^k,\nu_\tau^k)-\langle f_\tau^k,v^{h,k}-\nu_\tau^k\rangle
+\frac{1}{\tau}(\nu_\tau^{h,k-1}-\nu_\tau^{k-1},\nu_\tau^{h,k}-\nu_\tau^k)
+(z_\tau^{h,k},v^{h,k}-\nu_\tau^k).
\end{eqnarray*}
Let $e_k=\nu_\tau^{h,k}-\nu_\tau^k$, $g_k=u_\tau^{h,k}-u_\tau^{k}$ and $l_k=v^{h,k}-\nu_\tau^k$. Then it follows from \eqref{pro:pro1}-\eqref{pro:pro7} that
\begin{eqnarray*}\label{next1}
\frac{1}{\tau}\|e_k\|_H^2+M_B\|e_k\|^2
&\leq&  L_A\|u_\tau^{h,k}\|\|l_k\|
+\frac{1}{2\tau}\langle Ag_{k-1},g_{k-1}\rangle
-\frac{1}{2\tau}\langle Ag_k, g_k\rangle
-\frac{1}{2}\langle A e_k, e_k\rangle
\nonumber\\
&&\mbox{}
+L_j\|e_k\|\|v^{h,k}-\nu_\tau^{h,k}\|+L_B\|\nu_\tau^{h,k}\|\|l_k\|+C_j\|\nu_\tau^k\|\|l_k\|\nonumber\\
&&\mbox{}+\|f_\tau^k\|\|l_k\|
+\frac{1}{\tau}(e_{k-1},e_k)
+\|z_\tau^{h,k}\|_H\|l_k\|_H
\end{eqnarray*}
and so
\begin{eqnarray}\label{e96}
&&\frac{1}{2\tau}\|e_k\|_H^2-
\frac{1}{2\tau}\|e_{k-1}\|_H^2+M_B\|e_k\|^2
-\frac{1}{2\tau}\langle Ag_{k-1},g_{k-1}\rangle
+\frac{1}{2\tau}\langle Ag_k, g_k\rangle
+\frac{M_A}{2}\| e_k\|^2\nonumber\\
&\leq&  L_A\|u_\tau^{h,k}\|\|l_k\|
+L_j\|e_k\|\|v^{h,k}-\nu_\tau^{h,k}\|+L_B\|\nu_\tau^{h,k}\|\|l_k\|+C_j\|\nu_\tau^k\|\|l_k\|\nonumber\\
&&\mbox{}+\|f_\tau^k\|\|l_k\|
+\|z_\tau^{h,k}\|_H\|l_k\|_H.
\end{eqnarray}
For the term $\|v^{h,k}-\nu_\tau^{h,k}\|$ in \eqref{e96}, one has
\begin{eqnarray*}
\|v^{h,k}-\nu_\tau^{h,k}\|\leq \|v^{h,k}-\nu_\tau^{k}\|+\|\nu_\tau^{h,k}-\nu_\tau^{k}\|.
\end{eqnarray*}
It follows from \eqref{e96} that
\begin{eqnarray}\label{my8}
&&\frac{1}{2\tau}\|e_k\|_H^2-
\frac{1}{2\tau}\|e_{k-1}\|_H^2+(M_B-L_j)\|e_k\|^2
-\frac{1}{2\tau}\langle Ag_{k-1},g_{k-1}\rangle
+\frac{1}{2\tau}\langle Ag_k, g_k\rangle
+\frac{M_A}{2}\| e_k\|^2\nonumber\\
&\leq&  L_A\|u_\tau^{h,k}\|\|l_k\|
+L_j\|e_k\|\|l_k\|+L_B\|\nu_\tau^{h,k}\|\|l_k\|+C_j\|\nu_\tau^k\|\|l_k\|+\|f_\tau^k\|\|l_k\|
+\|z_\tau^{h,k}\|_H\|l_k\|_H.
\end{eqnarray}
Multiplying both sides of \eqref{my8} by $2\tau$, and then summing up these inequalities from 1 to $n$ with $1\leq n \leq N$,
we obtain 
\begin{eqnarray}\label{my3}
&&\|e_n\|_H^2+2\tau(M_B-L_j)\sum_{1\leq k\leq n}\|e_k\|^2+M_A\| g_n\|^2
+\tau M_A\sum_{1\leq k\leq n}\| e_k\|^2\nonumber\\
&\leq& \|e_0\|_H^2+L_A\|g_{0}\|^2
+2\tau\sum_{1\leq k\leq n} \left(L_A\|u_\tau^{h,k}\|+ L_j\|e_k\|+L_B\|\nu_\tau^{h,k}\|+C_j\|\nu_\tau^k\|+\|f_\tau^k\|\right)\|l_k\|\nonumber\\
&&\mbox{}+2\tau\sum_{1\leq k\leq n} \|z_\tau^{h,k}\|_H\|l_k\|_H.
\end{eqnarray}
Noting that $V$ is continuously embedded in $H$, we can see that $\|\cdot\|_H \leq C\|\cdot\|_V$.
From Lemmas \ref{lem:usd} and \ref{lem:Rothe2}, it is easy to know that there exists a constant $C$ being independent of $\tau$ such that
\begin{eqnarray}\label{my2}
\max_{1\le k\le N}\|z_{\tau}^{h,k}\|_H\leq C,\quad \max_{1\le k\le N}\|\nu_{\tau}^{h,k}\|\leq C,\quad
\sum_{k=1}^{N}\|\nu_{\tau}^{h,k}-\nu_{\tau}^{h,k-1}\|\leq C, \quad \sum_{k=2}^{N}\|z_{\tau}^{h,k}-z_{\tau}^{h,k-1}\|_H\leq C.
\end{eqnarray}
Thus, it follows from  \eqref{my3} and Lemma \ref{lem:Grownwell} that
\begin{equation}\label{next2}
\begin{aligned}
\max_{1\leq n\leq N}\| g_n\|^2\leq C\left(\max_{1\leq n\leq N}\{\|l_n\|\}+\|e_0\|^2+\|g_{0}\|^2\right).
\end{aligned}
\end{equation}
Let $v^{h,k}=P_h(\nu^{k}_\tau)$, where $P_h$ is the projection operator from $V$ to $V^h$. Then it follows from Proposition 6.2 of \cite{HanSonfonea 2000} that $\|l_k\|_V\sim\mathbf{O}(h)$.
Therefore, we can perform the 1-order approximation for Problem \ref{problem:Rothe} by using the finite element method.

\section*{Acknowledgements}
The authors are grateful to Professor M. Sofonea for his valuable comments and suggestions.

\end{document}